\documentclass[9pt]{amsart}
\usepackage{amsmath}
\usepackage{amsthm}
\usepackage{amsbsy}
\usepackage{amssymb}
\usepackage{amsfonts}
\usepackage[dvips]{lscape}
\usepackage{xcolor}
\usepackage{amscd}
\usepackage[all,cmtip]{xy}
\usepackage{euscript}
\usepackage{parskip}
\usepackage{enumerate}
\usepackage[all,cmtip]{xy}
\usepackage{tikz-cd}

\usepackage[margin=1in]{geometry}
\usepackage[expansion=false]{microtype}

\usepackage[pdfusetitle,unicode,hidelinks]{hyperref}

\theoremstyle{plain}

\theoremstyle{plain}
\newtheorem{theorem}{Theorem}[section]

\newtheorem{proposition}[theorem]{Proposition}
\newtheorem{lemma}[theorem]{Lemma}
\newtheorem{corollary}[theorem]{Corollary}

\newtheorem{conjecture}[theorem]{Conjecture}

\theoremstyle{remark}
\newtheorem{remark}[equation]{Remark}

\newtheorem{question}[theorem]{Question}

\theoremstyle{definition}
\newtheorem{definition}[theorem]{Definition}

\newcommand{\cal}{\EuScript}
\newcommand{\Spec}{\textup{Spec}}

\newcommand{\Sym}{\textup{Sym}}

\newcommand{\Hom}{\textup{Hom}}

\newcommand{\id}{\textup{id}}

\newcommand{\supp}{\textup{supp}}
\let\lim=\relax

\DeclareMathOperator*{\lim}{lim}

\newcommand{\diag}{\textup{diag}}

\newcommand{\tr}{\textup{tr}}

\renewcommand{\contentsname}{}

\newcommand{\PSL}{\textup{PSL}}
\newcommand{\SL}{\textup{SL}}

\newcommand{\GL}{\textup{GL}}
\newcommand{\End}{\textup{End}}
\newcommand{\vol}{\textup{vol}}
\newcommand{\fix}{\textup{fix}}

\newcommand{\area}{\textup{area}}
\newcommand{\blambda}{\underline{\lambda}}

\begin{document}
\title{Random flat bundles and equidistribution}
\author{Masoud Zargar}
\address{Max Planck Institute for Mathematics, 53111 Bonn, Germany}
\address{Department of Mathematics, University of Southern California, 3620 S. Vermont
Ave., Los Angeles, CA 90089-2532, U.S.A.}
\email{mzargar@usc.edu}
\renewcommand{\contentsname}{}
\maketitle
\vspace{-0.5cm}
\begin{center}
\end{center}
\begin{abstract}
Each signature $\blambda(n)=(\lambda_1(n),\hdots,\lambda_n(n))$, where $\lambda_1(n)\geq\dots\geq\lambda_n(n)$ are integers, gives an irreducible representation $\pi_{\blambda(n)}:U(n)\rightarrow\GL(V_{\blambda(n)})$ of the unitary group $U(n)$. Suppose $X$ is a finite-area cusped hyperbolic surface, $\chi$ is a random surface representation in $\Hom(\pi_1(X),U(n))$ equipped with a Haar unitary probability measure, and $(\blambda(n))_{n=1}^{\infty}$ is a sequence of signatures. Let $|\blambda(n)|:=\sum_i|\lambda_i(n)|$. We show that there is an absolute constant $c>0$ such that if $0\neq |\blambda(n)|\leq c\frac{\log n}{\log\log n}$ for sufficiently large $n$, then the Laplacians $\Delta_{\chi,\blambda(n)}$ acting on sections of the flat unitary bundles associated to the surface representations
\[\pi_1(X)\xrightarrow{\chi} U(n)\xrightarrow{\pi_{\blambda(n)}}\GL(V_{\blambda(n)})\]
have the property that for every $\varepsilon>0$ 
\[\mathbb{P}\left[\chi:\inf\Spec(\Delta_{\chi,\blambda(n)})\geq\frac{1}{4}-\varepsilon\right]\xrightarrow{n\rightarrow\infty}1,\]
where $\Spec(\Delta_{\chi,\blambda(n)})$ is the spectrum of $\Delta_{\chi,\blambda(n)}$. A special case of this is that flat unitary bundles associated to $\chi:\pi_1(X)\rightarrow U(n)$ asymptotically almost surely as $n\rightarrow\infty$ have least eigenvalue at least $\frac{1}{4}-\varepsilon$, irrespective of the spectral gap of $X$ itself. This is proved using the Hide--Magee method~\cite{HM}. The general spectral theorem above leads to a probabilistic equidistribution theorem for the images of closed geodesics under surface representations $\chi$. Using the spectral theorem above and proving a probabilistic prime geodesic theorem, we also obtain a probabilistic equidistribution theorem for the images under $\chi$ of geodesics of lengths dependent on the rank $n$.
\end{abstract}
\setcounter{tocdepth}{1}
\tableofcontents
\section{Introduction}\label{intro}
A topic of great interest in mathematics, physics, and computer science is the study of spectral gaps of graphs and Riemannian manifolds. Those finite simple $d$-regular graphs whose eigenvalues different from $\pm d$ are bounded in absolute value from above by $2\sqrt{d-1}$ are known as Ramanujan graphs. Infinite families of $d$-regular Ramanujan graphs are known to exist. For prime powers $d-1$, they have been explicitly constructed by Margulis~\cite{Margulis}, Lubotzky--Phillips--Sarnak~\cite{LPS}, and Morgenstern~\cite{Morgenstern} using number-theoretic methods. For other values of $d$, using probabilistic methods combined with interlacing polynomials, Marcus--Spielman--Srivastava~\cite{MSS} have proved the existence of infinite families of $d$-regular Ramanujan graphs.\\
\\
$2\sqrt{d-1}$ is the spectral radius of the universal covering tree of a $d$-regular graph. For hyperbolic surfaces, their universal cover $\mathbb{H}$ plays an analogous role, with its Laplacian having least eigenvalue $\frac{1}{4}$ as the analogue of $2\sqrt{d-1}$ for graphs. As of the writing of this paper, it is not known if infinite families of closed connected hyperbolic surfaces with enlarging genus exist each of whose Laplace-Beltrami operators have least positive eigenvalue, i.e. spectral gap, at least $\frac{1}{4}$. However, a recent result of Hide--Magee~\cite{HM} shows that infinite families of closed connected hyperbolic surfaces with enlarging genera exist with spectral gaps \textit{converging} to $\frac{1}{4}$. Their proof is probabilistic and, in contrast to graphs, no constructive proof is known as of the writing of this paper---even for this near-optimal spectral gap result. The theorem proved by Hide--Magee~\cite{HM} should be compared to a conjecture and theorem of Friedman~\cite{Fri03} stating that random covers of any fixed graph are almost Ramanujan with high probability. Recently, a new simpler proof of Friedman's conjecture was given Bordenave--Collins~\cite{BC19} using new methods. On the other hand, one of the central conjectures in number theory, Selberg's conjecture, states the following.
\begin{conjecture}[Selberg's conjecture]\label{Selbergconj}
Arithmetic hyperbolic surfaces have spectral gaps $\lambda_1$ at least $\frac{1}{4}$.
\end{conjecture}
It is known that Conjecture~\ref{Selbergconj} implies that there should be an infinite sequence of connected closed hyperbolic surfaces with increasing genera each of which has a spectral gap of at least $\frac{1}{4}$. However, Selberg's conjecture~\ref{Selbergconj} is beyond the reach of current methods. The best known spectral gap for arithmetic hyperbolic surfaces is due to Kim--Sarnak~\cite{KimSarnak}:
\[\lambda_1\geq\frac{1}{4}-\left(\frac{7}{64}\right)^2.\]
Spectral gaps are closely related to equidistribution results, for example by the prime geodesic theorem. A goal of this paper is the proof of a probabilistic equidistribution theorem for images of geodesics on cusped hyperbolic surfaces under unitary surface representations. This is intimately connected to the smallest eigenvalue of the Laplacian acting on sections of families of flat unitary bundles on such surfaces.\\
\\
Our setup is as follows. Suppose $X$ is a non-compact, complete, connected, finite area hyperbolic surface. 
\[\textit{Henceforth, we simply refer to such surfaces as \underline{cusped} hyperbolic surfaces}.\] We consider surface representations $\chi:\pi_1(X)\rightarrow U(n)$ into the unitary group. Since $X$ is a non-compact surface, $\pi_1(X)$ is a free group, and so the space $\Hom(\pi_1(X),U(n))$ of surface representations may be identified with a product of copies of $U(n)$, one copy for each generator of $\pi_1(X)$. Therefore, we naturally endow $\Hom(\pi_1(X),U(n))$ with a Haar unitary probability measure. Every sequence of signatures $\blambda(n)=(\lambda_1(n),\hdots,\lambda_n(n))$, $n$ varying, where $\lambda_1(n)\geq\hdots\geq\lambda_n(n)$ are integers, gives an infinite family of irreducible unitary representations
\[\pi_{\blambda(n)}:U(n)\rightarrow\GL(V_{\blambda(n)}).\]
For each sequence of signatures $\blambda(n)$, we have an infinite family of normalized characters
\[\chi_{\blambda(n)}(-):=\frac{\tr(\pi_{\blambda(n)}(-))}{\dim V_{\blambda(n)}}.\]
The surface representation
\[\pi_1(X)\xrightarrow{\chi}U(n)\xrightarrow{\pi_{\blambda(n)}}\GL(V_{\blambda(n)})\]
gives rise to a flat unitary vector bundle $E_{\chi,\blambda(n)}$ over $X$ of rank $\dim V_{\blambda(n)}$. The Laplacian $\Delta_{\chi,\blambda(n)}$ acts on sections of $E_{\chi,\blambda(n)}$, whose smallest eigenvalue we want to study. We denote by $\Spec(\Delta_{\chi,\blambda(n)})$ the spectrum of $\Delta_{\chi,\blambda(n)}$. In this paper, we abuse language and speak of the eigenvalues of surface representations to mean the eigenvalues of the Laplacian acting on sections of the associated flat bundles. Our first theorem is the following analogue of~\cite{HM} concerning minimal eigenvalues for flat unitary vector bundles whose proof follows ideas of~\cite{HM} with the appropriate generalizations and modifications to the case of irreducible representations of the unitary groups. It serves as a lemma in the proof of the equidistribution Theorem~\ref{geodesicdep} below.
\begin{theorem}\label{Thm1}There is an absolute constant $c>0$ such that if $X$ is any cusped hyperbolic surface and $\blambda(n)$ is a sequence of signatures with $0\neq|\blambda(n)|\leq c\frac{\log n}{\log\log n}$ for sufficiently large $n$, for every $\varepsilon>0$ we have
\[\mathbb{P}\left[\chi:\inf\Spec(\Delta_{\chi,\blambda(n)})\geq\frac{1}{4}-\varepsilon\right]\xrightarrow{n\rightarrow\infty}1.\]
\end{theorem}
In particular, if our irreducible representations are the standard representations $U(n)\xrightarrow{\id}U(n)$ corresponding to the signatures
\[\blambda(n)=(1,\underbrace{0,\hdots,0}_{n-1\ \textup{times}}),\]
Theorem~\ref{Thm1} implies that asymptotically almost surely as $n\rightarrow\infty$, flat unitary rank $n$ bundles over $X$ have smallest eigenvalue at least $\frac{1}{4}-\varepsilon$, irrespective of the spectral gap of $X$ itself. The following corollary to Theorem~\ref{Thm1} follows from the prime geodesic theorem for flat unitary bundles.
\begin{corollary}\label{Cor1}Suppose $X$ is a cusped hyperbolic surface, and let $\blambda_1(n),\hdots,\blambda_k(n)$ be  $k$ sequences of signatures. Let $f_n:=\sum_{i=1}^ka_i\chi_{\blambda_i(n)}$, $a_i\in\mathbb{C}$ fixed. Then if for each $i$ we have $|\blambda_i(n)|\leq c\frac{\log n}{\log\log n}$ for $n$ large enough (where $c>0$ is the absolute constant in Theorem~\ref{Thm1}), then asymptotically almost surely as $n\rightarrow \infty$, $\chi:\pi_1(X)\rightarrow U(n)$ satisfies
\[\frac{1}{\pi_X(T)}\sum_{\ell(\gamma)\leq T}f_n(\chi(\gamma))=\int_{U(n)}f_n(g)d\mu(g)+O(e^{-T/4})\]
as $T\rightarrow\infty$. Here, $\pi_X(T)$ is the number of closed oriented hyperbolic geodesics $\gamma$ on $X$ of length $\ell(\gamma)\leq T$, and $d\mu$ is the probability Haar measure on $U(n)$.
\end{corollary}
In turn, we also obtain the following corollary by combining Corollary~\ref{Cor1} with the main result of Diaconis--Shahshahani~\cite{DS}.
\begin{corollary}\label{Cor2}Suppose $r\geq 1$ and $\boldsymbol{a}=(a_1,\hdots,a_r)$ and $\boldsymbol{b}=(b_1,\hdots,b_r)$ are $r$-tuples of non-negative integers. Then for $X$ any cusped hyperbolic surface, asymptotically almost surely as $n\rightarrow\infty$, $\chi:\pi_1(X)\rightarrow U(n)$ satisfies
\[\frac{1}{\pi_X(T)}\sum_{\ell(\gamma)\leq T}\prod_{j=1}^r\tr(\chi(\gamma^j))^{a_j}\overline{\tr(\chi(\gamma^j))}^{b_j}=\delta_{\boldsymbol{a},\boldsymbol{b}}\int_{\mathbb{C}^r}|z_1|^{2a_1}\hdots|z_r|^{2a_r}\prod_{j=1}^r\frac{1}{j}e^{-\pi|z_j|^2/j}dx_jdy_j+O(e^{-T/4})\]
as $T\rightarrow\infty$, where $\delta_{\boldsymbol{a},\boldsymbol{b}}$ is the Kronecker delta.
\end{corollary}
This says essentially that the $(\tr(\chi(\gamma)),\tr(\chi(\gamma^2)),\hdots,\tr(\chi(\gamma^r)))$ distribute according to the law of $r$ independent complex Gaussians with a good error term. The main term on the right hand side is the expectation of $\prod_{j=1}^r\tr(g^j)^{a_j}\overline{\tr(g^j)}^{b_j}$ over $U(n)$. Explicitly, we have
\[\int_{\mathbb{C}^r}|z_1|^{2a_1}\hdots|z_r|^{2a_r}\prod_{j=1}^r\frac{1}{j}e^{-\pi|z_j|^2/j}dx_jdy_j=\prod_{j=1}^rj^{a_j}a_j!.\]
A more difficult problem is the study of the rate at which equidistribution happens. This could be made precise by bounding the lengths of the geodesic in terms of $n$. We could ask the following question.
\begin{question}\label{question}Given signature $\blambda\neq 0$, is there a constant $C(X,\blambda)>0$ such that
\[\mathbb{P}\left[\chi:\pi_1(X)\rightarrow U(n):\left|\frac{1}{\pi_X(T(n))}\sum_{\ell(\gamma)<T(n)}\chi_{\blambda}(\chi(\gamma))\right|\leq C(X,\blambda)e^{-c(n)}\right]\xrightarrow{n\rightarrow \infty}1\]
for some functions $T(n),c(n)>0$ going to infinity as $n\rightarrow\infty$? What are the optimal such functions?
\end{question}
Note that given a signature $\blambda$ of length $m$, we may extend it by zeros to obtain signatures of lengths $n$ for every $n\geq m$. This is how a fixed signature $\blambda$ gives rise to an infinite family of irreducible unitary representations of $U(n)$ for large enough $n$.
\begin{definition}\label{unbalanced}A signature $\blambda$ is \textit{unbalanced} if the sum of its entries is nonzero. Otherwise, $\blambda$ is \textit{balanced}.
\end{definition}
In relation to Question~\ref{question}, we prove the following theorem whose proof occupies the bulk of this paper.
\begin{theorem}\label{geodesicdep}Suppose $X$ is a cusped hyperbolic surface, and suppose $\blambda(n)$ is sequence of unbalanced signatures such that $|\blambda(n)|\leq c\frac{\log n}{\log\log n}$ for large enough $n$, where $c>0$ is the universal constant in Theorem~\ref{Thm1}. Suppose $T(n)$ is a function of $n$ such that $T(n)\rightarrow\infty$ as $n\rightarrow\infty$. For every $\varepsilon>0$,
\begin{equation}
\mathbb{P}\left[\chi:\left|\frac{1}{\pi_X(\log T(n))}\sum_{\ell(\gamma)\leq \log T(n)}\chi_{\blambda(n)}(\chi(\gamma))\right|\leq n^2\dim V_{\blambda(n)}T(n)^{-\frac{1}{4}+\varepsilon}\right]\xrightarrow{n\rightarrow\infty} 1.
\end{equation}
\end{theorem}
A reason we impose the condition that the signatures $\blambda(n)$ are unbalanced is the following. For unbalanced signatures $\blambda(n)$, we show that almost surely, $\pi_{\blambda(n)}\circ\chi$ gives a flat bundle $E_{\chi,\blambda(n)}$ that is non-singular, that is, whose spectrum has no continuous part. For balanced signatures $\blambda(n)$ whose sum of entries is zero, the irreducible representations $\pi_{\blambda(n)}$ are such that for every $A\in U(n)$, $\pi_{\blambda(n)}(A)$ has $1$ as an eigenvalue. Since $X$ is cusped hyperbolic, this means that for such representations, we always have a continuous part in the spectrum of $E_{\chi,\blambda(n)}$. The existence of a continuous spectrum requires a study of the main term of the hyperbolic scattering determinants of such flat bundles. The study of the hyperbolic scattering determinant in this setting is closely related to currently intractable problems in random matrix theory. Furthermore, in order to obtain Theorem~\ref{geodesicdep} for balanced $\blambda(n)$, we also need a bound in terms of $X$ and $\blambda(n)$ on the real parts of the resonances of $E_{\chi,\blambda(n)}$, at least asymptotically almost surely as $n\rightarrow\infty$. This is also related to understanding the main term of the hyperbolic scattering determinants. We briefly discuss these complications in Section~\ref{comments}.
\subsection{Relation to other works}In Hide--Magee~\cite{HM}, the authors show that a cusped hyperbolic surface is such that, for any $\varepsilon>0$, if one takes $n$-sheeted covers of it, as $n\rightarrow\infty$, asymptotically almost surely they do not introduce \textit{new} eigenvalues less than $\frac{1}{4}-\varepsilon$. In particular, if one starts with a cusped hyperbolic surface with spectral gap at least $\frac{1}{4}$, then one obtains that $n$-sheeted covers asymptotically almost surely have spectral gaps at least $\frac{1}{4}-\varepsilon$. A compactification argument allows them to compactify their sequence of hyperbolic surfaces to show the existence of an infinite family of connected closed hyperbolic surfaces $\{X_n\}_n$ with increasing genera and spectral gaps $\lambda_1(X_n)$ satisfying
\[\limsup_{n\rightarrow\infty}\lambda_1(X_n)=\frac{1}{4}.\]
The existence of such a sequence of hyperbolic surfaces was an important open question in the spectral geometry of hyperbolic surface. Note that by a theorem of Huber~\cite{Huber}, any sequence $\{X_n\}_n$ of closed connected hyperbolic surfaces with increasing genera satisfies
\[\limsup_{n\rightarrow\infty}\lambda_1(X_n)\leq\frac{1}{4}.\]
Note that $n$-sheeted covers of a surface $X$ are parametrized by surface representations
\[\pi_1(X)\rightarrow S_n\]
into the symmetric group $S_n$ on $n$ elements. One could instead consider the larger family of rank $n$ flat bundles over $X$, essentially parametrized by surface representations
\[\chi:\pi_1(X)\rightarrow U(n)\]
into the unitary group $U(n)$ of rank $n$. If one wants to investigate how the images of (homotopy classes of oriented) closed geodesics $\gamma$ on $X$ under surface representations $\chi$ distribute, we are led to considering class functions $f:U(n)\rightarrow\mathbb{C}$ evaluated at $\chi(\gamma)$. Since class functions are generated by irreducible characters, we need to understand the minimal eigenvalue of the Laplacian acting on sections of the flat bundle associated to
\[\pi_1(X)\xrightarrow{\chi}U(n)\xrightarrow{\pi_{\blambda(n)}}\GL(V_{\blambda(n)}),\]
where $\pi_{\blambda(n)}$ are irreducible representations. Our first theorem, Theorem~\ref{Thm1}, addresses this minimal eigenvalue problem. Its proof follows that of the main theorem of~\cite{HM} with the necessary modifications and generalizations to our unitary setting. The proofs of the other results occupy the bulk of this paper and use other ideas.\\
\\
The problem considered in this paper is also related to the expected Wilson loops. Indeed, suppose $\gamma$ is a closed oriented geodesic on a hyperbolic surface $X$. One could consider the \textit{finite} space of surface representations
\[\chi:\pi_1(X)\rightarrow S_n\]
given the uniform probability measure. Then, one could ask about the quantity
\[\mathbb{E}_{\chi}[\fix_{\chi}(\gamma)]:=\frac{1}{|\Hom(\pi_1(X),S_n)|}\sum_{\chi}\fix_{\chi}(\gamma),\]
where $\fix_{\chi}(\gamma)$ is the number of fixed points of the permutation $\chi(\gamma)\in S_n$. In the case of closed orientable surfaces, Magee--Puder~\cite{MP20} proved the following result.
\begin{theorem}\cite[Thm. 1.2]{MP20}\label{MageePuder}
Suppose $X$ is a connected closed orientable surface. If $\gamma\in\pi_1(X)$ is non-trivial and $q\in\mathbb{N}$ is maximal such that $\gamma=\gamma_0^q$ for some $\gamma_0\in\pi_1(X)$, then as $n\rightarrow\infty$,
\[\mathbb{E}_{\chi}[\fix_{\chi}(\gamma)]=d(q)+O(1/n),\]
where $d(q)$ is the number of divisors of $q$.
\end{theorem}
This is a special case of a more general result they proved. This result was preceeded by a result in the case when $X$ is not closed, that is, $\pi_1(X)$ is a free group. In this case, Parzanchevski--Puder~\cite{PP15} proved the following.
\begin{theorem}\cite[Thm. 1.8]{PP15}
\[\mathbb{E}_{\chi}[\fix_{\chi}(\gamma)]=1+\frac{c(\gamma)}{n^{\pi(\gamma)-1}}+O(1/n^{\pi(\gamma)}),\]
where $\pi(\gamma)\in\{0,\hdots,r\}\cup\{\infty\}$, $r$ the rank of the free group $\pi_1(X)$, is an algebraic invariant of $\gamma$ called the primitivity rank, and $c(\gamma)\in\mathbb{N}$.
\end{theorem}
In Magee--Naud--Puder~\cite{MNP}, the following theorem was used in combination with Selberg's trace formula to prove that asymptotically almost surely as $n\rightarrow\infty$, $n$-sheeted covers of a closed hyperbolic surface introduce no new eigenvalues below $\frac{3}{16}-\varepsilon$.
\begin{theorem}\cite[Thm. 1.11]{MNP}Suppose $X$ is a closed orientable surface of genus $g\geq 2$. For each $g\geq 2$, there is a constant $A=A(g)$ such that for any $c>0$, if $1\neq \gamma\in\pi_1(X)$ is not a proper power of another element in $\pi_1(X)$ with word length $\ell_{w}(\gamma)\leq c\log n$ then
\[\mathbb{E}_{\chi}[\fix_{\chi}(\gamma)]=1+O_{c,g}\left(\frac{(\log n)^A}{n}\right).\]
\end{theorem}
One could ask for analogues of the above results for surface representations
\[\chi:\pi_1(X)\rightarrow U(n).\]
In this case, when $\pi_1(X)$ is free, one uses a uniform Haar probability measure. On the other hand, when $X$ is closed orientable, then one uses the Atiyah--Bott--Goldman symplectic volume form on the character variety of surface representations $\chi$. One could then study quantities of the form
\[\mathbb{E}_{\chi}[f_n(\chi(\gamma))],\]
where we are averaging over the surface representations $\chi$ and $f_n:U(n)\rightarrow\mathbb{C}$ are class functions. Magee~\cite{Magee21} has studied this question in the case of closed orientable surfaces of genus $g\geq 2$ for the trace function. In particular, he proved the following.
\begin{theorem}\cite[Cor. 1.2]{Magee21}Suppose $X$ is a closed orientable surface of genus $g\geq 2$. If $\gamma\in\pi_1(X)$, then the limit
\[\lim_{n\rightarrow\infty}\frac{\mathbb{E}_{\chi}[\tr(\chi(\gamma))]}{n}\]
exists.
\end{theorem}
The conjecture is that if $\gamma\neq 1$, then this limit is actually $0$. However, this is not known. In the case of cusped surfaces, however, a result of Voiculescu~\cite{Voiculescu} states the following.
\begin{theorem}\cite{Voiculescu}
Suppose $\pi_1(X)$ is free and $1\neq\gamma\in\pi_1(X)$. Then as $n\rightarrow\infty$,
\[\mathbb{E}_{\chi}[\tr(\chi(\gamma))]=o_{\gamma}(n).\]
\end{theorem}
In the case of non-closed surfaces, Magee--Puder~\cite{MP19} have related the asymptotic expression for such quantities to the topology of the surface. If we consider more general families of class functions $f_n:U(n)\rightarrow\mathbb{C}$, then the problem is much more difficult and very little is known.\\
\\
In the above results, the curve $\gamma\in\pi_1(X)$ was fixed and the average was taken over the different spaces of surface representations. However, one could instead fix the surface representation and average over a family of closed geodesics on a hyperbolic surface $X$. The theorems proved in this paper are the generalized analogues of some of the above theorems in this orthogonal regime.
\subsection{Outline of the paper}In Section~\ref{prelim}, we discuss preliminaries on representation theory, random matrix theory, flat bundles, and Selberg's trace formula. In Section~\ref{minimal}, we prove Theorem~\ref{Thm1} on the minimal eigenvalue of the Laplacian acting on sections of flat unitary bundles. The proof mostly follows that of the main theorem of~\cite{HM} with the necessary modifications and minor deviations. In Section~\ref{random}, we prove results in random matrix theory that allow us to restrict to flat bundles whose spectra do not have continuous parts. In Section~\ref{counting}, we prove probabilistic Weyl-law type results counting eigenvalues of the Laplacian for flat bundles. In Section~\ref{pgt}, we prove Theorem~\ref{geodesicdep}, a probabilistic prime geodesic theorem whose error term is explicit in terms of the ranks of the bundles. In Section~\ref{applications}, we deduce some equidistribution consequences of our spectral Theorem~\ref{Thm1}. In Section~\ref{comments}, we comment on the balanced case and discuss some of its relations to very difficult problems in random matrix theory.
\subsection{Acknowledgments}This project was supported by Max Planck Institute for Mathematics in Bonn and University of Southern California.

\section{Preliminaries}\label{prelim}
In this section, we briefly discuss preliminaries on flat bundles, representation theory of $U(n)$, random matrix theory, and Selberg's trace formula.
\subsection{Flat bundles}\label{flatbundlespre}Suppose $X$ is a surface with universal cover the upper half plane
\[\mathbb{H}:=\left\{z\in\mathbb{C}|\Im(z)>0\right\}\]
which may be equipped with the hyperbolic metric
\[ds^2=\frac{dx^2+dy^2}{y^2}.\]
$X$ also inherits this metric, thus making it into a \textit{hyperbolic} surface. Throughout this paper, our surfaces $X$ are non-compact, connected, hyperbolic, complete, and of finite area. As previously mentioned, we refer to such surfaces simply as \textit{cusped} hyperbolic surfaces. For such $X$, $\pi_1(X)$ is a finitely generated \textit{free} group. Given any surface representation
\[\chi:\pi_1(X)\rightarrow U(n),\]
we have a flat unitary vector bundle of rank $n$ constructed as follows. $\pi_1(X)$ acts on the universal cover $\mathbb{H}$ by deck transformations. Furthermore, $\chi$ gives us a unitary action of $\pi_1(X)$ on $\mathbb{C}^n$. Therefore, we may consider
\[E_{\chi}:=\mathbb{H}\times_{\chi}\mathbb{C}^n\]
which is $\mathbb{H}\times\mathbb{C}^n$ modulo the action
\[\gamma\cdot(z,v):=(\gamma\cdot z,\chi(\gamma)v)\]
for each $\gamma\in\pi_1(X)$. We may consider the space of global sections of the vector bundle $E_{\chi}\rightarrow X$. These are equivalently the space of functions $s:\mathbb{H}\rightarrow\mathbb{C}^n$ satisfying the automorphy condition
\begin{equation}\label{automorphy}s(\gamma\cdot z)=\chi(\gamma)s(z)\end{equation}
for every $\gamma\in\pi_1(X)$ and every $z\in\mathbb{H}$.\\
\\
Of central importance to us are the spectral properties of the flat unitary vector bundles over $X$. The Laplacian on the upper halfplane is the differential operator
\[\Delta=-y^2\left(\frac{\partial^2}{\partial x^2}+\frac{\partial^2}{\partial y^2}\right).\]
This may be extended coordinate-wise to a differential operator on the space of $C^2$-functions $\mathbb{H}\rightarrow\mathbb{C}^n$. This operator is invariant under the action of $\PSL_2(\mathbb{R})$ on $\mathbb{H}$ by M\"obius transformations; this is the group of orientation-preserving isometries of $\mathbb{H}$. By the discussion above, we then have a Laplacian $\Delta_{\chi}$ that acts on the space of $C^{\infty}$ global sections $C^{\infty}(X;E_{\chi})$ of $E_{\chi}\rightarrow X$ once we view each section as a function $\mathbb{H}\rightarrow\mathbb{C}^n$ with the automorphy condition~\eqref{automorphy}.\\
\\
Since $E_{\chi}$ is a unitary bundle, we may define the $L^2$-norm
\begin{equation}\label{L2norm}\|s\|_{L^2}:=\left(\int_X|s(z)|^2\frac{dxdy}{y^2}\right)^{\frac{1}{2}}\end{equation}
for each smooth global section $s\in C^{\infty}(X;E_{\chi})$. Note that this is well-defined as unitarity implies that for each $\gamma\in\pi_1(X)$ and each $z\in\mathbb{H}$, $|s(\gamma\cdot z)|=|\chi(\gamma)s(z)|=|s(z)|$, that is, $z\mapsto |s(z)|$ is a function on $X$. 
We denote by $L^2(X;E_{\chi})$ the completion of the space $C_c^{\infty}(X;E_{\chi})$ of \textit{compactly supported} $C^{\infty}$ global sections of $E_{\chi}$ with respect to the $L^2$-norm~\eqref{L2norm}.\\
\\
In general, given a flat unitary complex vector bundles $E_{\chi}\rightarrow X$, and global sections $s,t:X\rightarrow E_{\chi}$, the map $\mathbb{H}\rightarrow\mathbb{C}$ given by $x\mapsto\left<s(x),t(x)\right>$, the fiberwise Hermitian product, is a function on $X$; for each $\gamma\in\pi_1(X)$, $\left<s(\gamma x),t(\gamma x)\right>=\left<\chi(\gamma)s(x),\chi(\gamma)t(x)\right>=\left<s(x),t(x)\right>$. We simply write $\left<s,t\right>_{E_{\chi}}:X\rightarrow\mathbb{C}$ for this function.\\
\\ 
We may extend the Laplacian $\Delta_{\chi}$ to a differential operator on $L^2(X;E_{\chi})$ in the following sense. Endow $C^{\infty}_c(X;E_{\chi})$ with the Sobolev norm $\|.\|_{W^2}$ given by
\[\|s\|^2_{W^2}:=\|s\|^2_{L^2}+\|\Delta_{\chi}s\|^2_{L^2}.\]
The Sobolev space $W^2(X;E_{\chi})$ is the completion of $C^{\infty}_c(X;E_{\chi})$  with respect to $\|.\|_{W^2}$. The Laplacian $\Delta_{\chi}$ may then be extended to an operator
\[\Delta_{\chi}:L^2(X;E_{\chi})\rightarrow L^2(X;E_{\chi})\]
with domain the Sobolev space $W^2(X;E_{\chi})$.\\
\\
In order to do probability theory on the space of flat bundles, we note that since $\pi_1(X)$ is a finitely generated free group, say on $r$ generators, the space of surface representations $\Hom(\pi_1(X),U(n))$ may be identified with $U(n)^r$. We may consequently naturally endow this space with a probability Haar measure.
\subsection{Representation theory of $U(n)$} For details regarding the material in this subsection, see~\cite{DBump}. By the theory of highest weights, there is a bijection between the equivalence classes of irreducible unitary representations of $U(n)$ and their highest weights, which may be viewed as $n$-tuples of integers $\lambda_1\geq\dots\geq\lambda_n$. We may view them as linear combinations
\[\lambda_1\hat{\omega}_1+\dots+\lambda_n\hat{\omega}_n,\]
where the $\hat{\omega}_i$ are fundamental weights of $U(n)$. $\blambda:=(\lambda_1,\hdots,\lambda_n)$ is called the \textit{signature} of the representation. We write
\[|\blambda|:=\sum_i|\lambda_i|.\]
It will also be convenient for us to write $\blambda$ in the form
\[m_1\omega_1+\dots+m_n\omega_n,\]
where $\omega_i:=(1,\dots,1,0,\dots,0)$ with $i$ $1$s, and $m_i\in\mathbb{Z}_{\geq 0}$ for $1\leq i\leq n-1$ while $m_n\in\mathbb{Z}$. Given a signature $\blambda$, we denote by $l(\blambda)$ the number of nonzero entries of $\blambda$. Given a fixed signature $\blambda$, each $n\geq l(\blambda)$ gives rise to an irreducible unitary representation
\[\pi_{\blambda}:U(n)\rightarrow\GL(V_{\blambda})\]
with normalized character
\[\chi_{\blambda}(-):=\frac{\tr(\pi_{\blambda}(-))}{\dim V_{\blambda}}.\]
Writing $\blambda=m_1\omega_1+\dots+m_n\omega_n$, it is well-known that the representation $\pi_{\blambda}$ is an irreducible subrepresentation of the representation
\[(\mathbb{C}^n)^{\otimes m_1}\otimes(\wedge^2\mathbb{C}^n)^{\otimes m_2}\otimes\dots\otimes(\wedge^n\mathbb{C}^n)^{\otimes m_n}.\]
Of interest to us later in this paper are the power sum functions
\[P_j(x_1,\hdots,x_n):=\sum_{i=1}^nx_i^j.\]
For a partition $\lambda=1^{a_1}2^{a_2}\dots k^{a_k}$, we let $P_{\lambda}:=\prod_{j=1}^kP_j^{a_j}.$ Clearly, $\lambda$ is a partition of $a_1+2a_2+\hdots+ka_k=:K$. For $\lambda$ varying through all partitions of $K$, $P_{\lambda}$ form a basis of homogeneous symmetric polynomials of degree $K$ in $n$ variables, $n\geq K$. There is another basis for this vector space known as Schur functions $s_{\mu}$. There is a change of basis formula given by
\[P_{\lambda}=\sum_{\mu\vdash K}\chi_{\lambda}(\mu)s_{\mu},\]
where $\chi_{\lambda}(\mu)$ is the character of the symmetric group $S_K$ associated to the partition $\lambda$ evaluated on the conjugacy class associated to $\mu$. It is well-known that the Schur functions $s_{\mu}$ give characters of the unitary groups $U(n)$. In fact, if $A\in U(n)$ has eigenvalues $\exp(i\theta_j)$, then $s_{\mu}(A):=s_{\mu}(\exp(i\theta_1),\hdots,\exp(i\theta_n))$ is a character of $U(n)$ for each partition $\mu\vdash K$, $n\geq K$.\\
\\
In his paper, we will use the following Weyl integration formula for $U(n)$ when bounding eigenvalues of random unitary matrices away from $1$ in Proposition~\ref{logbound}.
\begin{theorem}[Weyl integration formula for $U(n)$]Suppose $f:U(n)\rightarrow\mathbb{C}$ is a continuous class function, and let $U(n)$ and the torus $\mathbb{T}^n$ be equipped with probability Haar measures $dg$ and $dt$, respectively. Then
\begin{equation}\label{Weylintegration}\int_{U(n)}f(g)dg=\frac{1}{n!}\int_{\mathbb{T}^n}f(\diag(t_1,\hdots,t_n))\prod_{i<j}|t_i-t_j|^2dt.\end{equation}
\end{theorem}
See Proposition 18.4 of~\cite{DBump}.
\subsection{Random matrix theory}Our proof of Theorem~\ref{Thm1} uses results in the random matrix theory of $U(n)$. We recall some results that we will use in this paper.\\
\\
Consider the right regular representation
\[\pi_{\infty}:\Gamma\rightarrow\End(\ell^2(\Gamma)),\]
where $\Gamma:=\pi_1(X)$ is freely generated by $\gamma_1,\hdots,\gamma_k$. A powerful result of Bordenave--Collins~\cite{BC} is the following. It states that non-trivial irreducible representations of $U(n)$ are almost surely strongly asymptotically free.
\begin{theorem}\label{ThmBC}\cite[Cor. 3]{BC}
Suppose $P_1,\hdots,P_{\ell}\in\mathbb{C}\left<X_1,\hdots,X_k,X_1^*,\hdots,X_k^*\right>$ are non-commutative complex $*$-polynomials, and suppose $a_1,\hdots,a_{\ell}\in M_r(\mathbb{C})$. Let $\blambda(n)$ be a sequence of signatures. Then there is an absolute constant $c>0$ such that if $0\neq |\blambda(n)|\leq c\frac{\log(n)}{\log(\log(n))}$ for sufficiently large $n$, then for any $\varepsilon>0$, asymptotically almost surely as $n\rightarrow\infty$,
\[\left\|\sum_{i=1}^{\ell}a_i\otimes P_i(\pi_{\blambda(n)}(U_1),\hdots,\pi_{\blambda(n)}(U_k))\right\|_{\mathbb{C}^r\otimes V_{\blambda(n)}}\leq \left\|\sum_{i=1}^{\ell}a_i\otimes P_i(\pi_{\infty}(\gamma_1),\hdots,\pi_{\infty}(\gamma_k))\right\|_{\mathbb{C}^r\otimes\ell^2(\Gamma)}+\varepsilon,\]
where the norms are operator norms.
\end{theorem}
We will use this theorem when bounding the norms of random operators appearing in the proof of Theorem~\ref{Thm1} in Section~\ref{minimal}. Theorem~\ref{ThmBC} as stated seems more general than Corollary 3 of~\cite{BC} in the sense that we are allowing arbitrary polynomials. However, the above consequence is well-known and follows from the standard linearization trick of Pisier~\cite{Pisier}, as noted in~\cite{BC} following Corollary 3 of~\cite{BC}.
\subsection{Spectral theory and Selberg's trace formula}\label{spectraltheory}
In our proof of equidistribution theorems we use the Selberg trace formula for flat unitary bundles over cusped hyperbolic surfaces $X$. Let $\chi:\Gamma=\pi_1(X)\rightarrow U(n)$ be a surface representation. We want to study the spectrum of $\Delta_{\chi}\curvearrowright L^2(X;E_{\chi})$.\\
\\
Suppose $\eta_1,\hdots,\eta_k$ are the cusps of $X$, to which we associate parabolic elements $T_1,\hdots,T_k\in\Gamma$, respectively, such that for each $j$ the stabilizer $\Gamma_{\eta_j}:=\{P\in\Gamma|P\eta_j=\eta_j\}=\left<T_j\right>$ is the cyclic group generated by $T_j$. For each $j$, $\chi(T_j)\in U(n)$ is a unitary matrix which may or may not have $1$ as an eigenvalue. Let $m_j\geq 0$ be the multiplicity of the eigenvalue $1$ for $\chi(T_j)$. The operator $\Delta_{\chi}$ is self-adjoint, and so it has a real (in fact, non-negative) spectrum, part of which is discrete, and another part of which is continuous (if it exists). It is well-known that $\Delta_{\chi}$ has no discrete spectrum if and only if $K_0:=\sum_jm_j=0$. 
\begin{remark}\textup{Though $X$ being non-compact has the advantage that $\pi_1(X)$ is a free group, and so the space of flat unitary bundles is easily described, $X$ non-compact introduces analytic complications due to the presence of a continuous spectrum when $K_0\geq 1$. On the other hand, the analogue of Theorem~\ref{ThmBC} is not known if we impose that relations need to be satisfied. Therefore, for compact $X$, though the analysis of the spectrum is simpler, the required random matrix theory results are not known. In particular, in the compact case, the probability measure on the space of surface representations should come from the Atiyah--Bott--Goldman symplectic volume form, and integration over this symplectic manifold is a very difficult and interesting problem.}
\end{remark}
We denote the discrete spectrum of $\Delta_{\chi}$ by
\[0\leq\lambda_0\leq\lambda_1\leq\lambda_2\leq\dots,\]
each eigenvalue appearing as many times as its multiplicity. We will use the following notation: $\lambda_n=s_n(1-s_n)=\frac{1}{4}+r_n^2$, where $s_n:=\frac{1}{2}-ir_n$ and $r_n:=\sqrt{\lambda_n-\frac{1}{4}}$. We let $\tilde{s}_n:=\frac{1}{2}+ir_n$. If $0\leq\lambda_n<\frac{1}{4}$, then $s_n\in(\frac{1}{2},1]$; otherwise, $s_n\in[\frac{1}{2},\frac{1}{2}+i\infty)$. The eigensections corresponding to these eigenvalues are called Maass forms. $\lambda_0=0$ if and only if there is a non-zero constant global section of $E_{\chi}\rightarrow X$. For the bundles under consideration in this paper, we show in Section~\ref{random} that almost surely there are no non-zero constant global sections. Therefore, it is natural to ask about the smallest (positive) eigenvalue of flat unitary bundles. When $\blambda$ is unbalanced, our flat bundles will almost surely be non-singular, that is, $K_0=0$.
\\
\\
An object that will play an important role in the proof of Theorem~\ref{geodesicdep} is Selberg's zeta function associated to flat unitary bundles $E_{\chi}$. Let $\Gamma_{\textup{hyp}}$ be the $\Gamma$-conjugacy classes of hyperbolic elements of $\Gamma$, and let $P\Gamma_{\textup{hyp}}$ be the $\Gamma$-conjugacy classes of \textit{primitive}, that is, not a proper power of another element, hyperbolic elements of $\Gamma$. For $P\in\Gamma_{\textup{hyp}}$ let $N(P)$ be the norm of $P$. Then \textit{Selberg's zeta function} for $\chi:\Gamma\rightarrow U(n)$ is given by
\[Z_{\chi}(s):=\prod_{P_0\in P\Gamma_{\textup{hyp}}}\prod_{k=0}^{\infty}\det(I_n-\chi(P_0)N(P_0)^{-k-s}).\]
We also have the associated \textit{Dirichlet series} given by the logarithmic derivative of Selberg's zeta function:
\[D_{\chi}(s):=\frac{Z'_{\chi}(s)}{Z_{\chi}(s)}=\sum_{P\in\Gamma_{\textup{hyp}}}\frac{\Lambda(P)\tr(\chi(P))}{N(P)^s},\]
where
\begin{equation}
\Lambda(P):=\frac{\log N(P_0)}{1-N(P)^{-1}}
\end{equation}
in which $P_0\in P\Gamma_{\textup{hyp}}$ is the primitive element associated to $P\in\Gamma_{\textup{hyp}}$.
With the above notation, Selberg's trace formula when $K_0=0$ states the following.
\begin{theorem}\cite[Thm. 4.2, p. 314]{Hejhal2}\label{Selbergtraceformula}Suppose $K_0=0$. Let $h:\mathbb{C}\rightarrow\mathbb{C}$ be a function satisfying the following hypotheses:
\begin{enumerate}[(a)]
	\item $h(r)$ is analytic on $|\Im(r)|\leq\frac{1}{2}+\delta$ for some $\delta>0$;
	\item $h(-r)=h(r)$;
	\item $|h(r)|=O((1+|\Re(r)|)^{-2-\delta})$.
\end{enumerate}
Suppose further that
\[g(u):=\frac{1}{2\pi}\int_{-\infty}^{\infty}h(r)e^{-iru}dr;\ u\in\mathbb{R}.\]
Then
\begin{eqnarray*}
\sum_kh(r_k)&=& \frac{n\area(X)}{4\pi}\int_{-\infty}^{\infty}rh(r)\tanh(\pi r)dr\\&+&
\sum_{P\in\Gamma_{\textup{hyp}}}\frac{\tr(\chi(P))\log N(P_0)}{N(P)^{1/2}-N(P)^{-1/2}}g(\log N(P))\\&-&
g(0)\sum_{i=1}^k\log|\det(I-\chi(T_i))|.
\end{eqnarray*}
\end{theorem}
Note that since $K_0=0$, $\det(I-\chi(T_i))\neq 0$ for every $i$.
\section{Minimal eigenvalues of flat unitary bundles}\label{minimal}
In this section, we prove Theorem~\ref{Thm1}. Suppose $\blambda(n)$ is a sequence of signatures such that $0\neq |\blambda(n)|\leq\frac{c\log n}{\log\log n}$ for $n$ large enough, where $c>0$ is the absolute constant appearing in Theorem~\ref{ThmBC} of Bordenave--Collins. Given a surface representation $\chi:\pi_1(X)=:\Gamma\rightarrow U(n)$, we have the surface representation
\[\Gamma\xrightarrow{\chi} U(n)\xrightarrow{\pi_{\blambda(n)}}\GL(V_{\blambda(n)}).\]
Associated to this surface representation, we have the flat rank $\dim V_{\blambda(n)}$ unitary bundle $E_{\chi,\blambda(n)}$ on whose sections the Laplacian $\Delta_{\chi,\blambda(n)}$ acts. We prove Theorem~\ref{Thm1} by showing that for each $s_0>\frac{1}{2}$, asymptotically almost surely over $\chi$ as $n\rightarrow\infty$, there is a bounded operator, called the \textit{resolvent operator},
\[R_{\chi,\blambda(n)}(s):L^2(X;E_{\chi,\blambda(n)})\rightarrow L^2(X;E_{\chi,\blambda(n)})\]
for any $s\in [s_0,1]$ such that
\[(\Delta_{\chi,\blambda(n)}-s(1-s))R_{\chi,\blambda(n)}(s)=\id_{L^2(X;E_{\chi,\blambda(n)})}.\]
When such a resolvent operator exists, $\Delta_{\chi,\blambda(n)}$ cannot have an eigenvalue less than $s_0(1-s_0)$.\\
\\
The way such a resolvent operator is constructed is by constructing such operators on the cuspidal and internal parts of $X$ and then gluing them together to obtain a resolvent operator on all of $X$. In~\cite{HM}, Hide--Magee construct such resolvents for finite covers in order to study the new eigenvalues one obtains by passing to covers. In their case, the defining representations of the symmetric groups $S_n$ play a role. We follow a similar proof in our setting of flat unitary bundles associated to the surface representations
\[\Gamma\xrightarrow{\chi} U(n)\xrightarrow{\pi_{\blambda(n)}}\GL(V_{\blambda(n)}).\]

\subsection{Cuspidal part}
In this subsection, we construct the part of the resolvent operator coming from the cuspidal part. For simplicity, we assume that $X$ has only one cusp. We model the cusp as
\[C:=(1,\infty)\times S^1\]
endowed with the metric $\frac{dr^2+dx^2}{r^2}$, where $x$ is the coordinate on $S^1$ and $r$ that on $(1,\infty)$. As in Hide--Magee~\cite{HM}, we begin by considering suitable smooth functions $\chi^{\pm}:(1,\infty)\rightarrow [0,1]$ that are identically zero in a neighbourhood of $r=1$, identical to the constant function $1$ in $(r_0,\infty)$ for some $r_0>1$, and satisfying
\begin{equation}\label{chiid}\chi^+\chi^-=\chi^-.\end{equation}
Naturally, these functions lift to functions $\chi_C^{\pm}:C\rightarrow [0,1]$ on the cusp $C$ via projection onto the $r$-coordinate, satisfying~\eqref{chiid}. Since $\chi_C^{\pm}$ are identially zero in a neighbourhood of the cusp $C$, we may extend these functions by zero to all of $X$, and obtain for each surface representation $\chi:\Gamma\rightarrow U(n)$ and each irreducible representation $\pi_{\blambda(n)}:U(n)\rightarrow\GL(V_{\blambda(n)})$, operators
\[\chi_{C,\chi}^{\pm}:L^2(X;E_{\chi,\blambda(n)})\rightarrow L^2(X;E_{\chi,\blambda(n)})\]
given by multiplication by functions $\chi_C^{\pm}$ on $X$. Replace $C$ with $(0,\infty)\times S^1$ equipped with the same hyperbolic metric and make the appropriate extensions by zero. 
\begin{lemma}\label{lowcusp}
Suppose $\omega\in W^2(C;E_{\sigma})$ is a section of the flat vector bundle $E_{\sigma}\rightarrow C$ coming from a representation $\sigma:\mathbb{Z}\simeq\pi_1(C)\rightarrow U(n)$. Then
\[\int_C\left<\Delta \omega,\omega\right>_{E_{\sigma}}dg\geq\frac{1}{4}\|\omega\|^2_{L^2(C;E_{\sigma})},\]
where $dg=\frac{dxdr}{r^2}$ is the Riemannian area form associated to the metric $\frac{dr^2+dx^2}{r^2}$. As a corollary, the Laplacian on sections of $E_{\sigma}$ has smallest eigenvalue at least $\frac{1}{4}$.
\end{lemma}
\begin{proof}
If the image of $\sigma$ is finite, then we may view the section $\omega$ as a function $\omega':C'\rightarrow\mathbb{C}^n$ on a finite covering $\pi:C'\rightarrow C$ satisfying the condition 
\[\omega'(\gamma\cdot x)=\sigma(\gamma)\omega'(x),\]
where $\gamma$ is the generator of $\pi_1(C)$. Furthermore, note that
\[\int_{C'}\left<\Delta \omega',\omega'\right>_{\mathbb{C}^n}dg=\deg\pi\int_C\left<\Delta \omega,\omega\right>_{E_{\sigma}}dg\]
and
\[\|\omega'\|^2_{L^2}=\int_{C'}\|\omega'\|^2_{\mathbb{C}^n}dg=\deg\pi\int_{C}\left<\omega,\omega\right>_{E_{\sigma}}dg=\deg\pi\|\omega\|^2_{L^2}.\]
This allows us to reduce the proof of the inequality in the finite image case to the case of functions which follows from Lemma 4.2 of Hide--Magee~\cite{HM}.\\
\\
When $\sigma$ does not have finite image, we may approximate $\omega$ by sections of bundles corresponding to representations with finite image. Indeed, after conjugation, we may assume without loss of generality that $\sigma(1)=\diag(e^{i\theta_1},\hdots,e^{i\theta_n})$. This is because unitary matrices are unitarily diagonalizable. Consider the model $\{(r,x,z)\in (1,\infty)\times\mathbb{R}\times\mathbb{C}^n\}$ of the universal cover of the flat bundle $E_{\sigma}$. The action of $\pi_1(C)$ on this model is determined by the action of $\sigma(1)$ given by
\[\sigma(1):(r,x,z)\mapsto \left(r,x+1,\diag(e^{i\theta_1},\hdots,e^{i\theta_n})z\right).\]
Suppose $\mathbf{\eta}:=(\eta_1,\hdots,\eta_n)$, $\eta_i\in\mathbb{R}$ such that $\eta_i+\theta_i\in\mathbb{Q}$ for every $i$. Define
\[p_{\mathbf{\eta}}:(r,x,z)\mapsto (r,x,\diag(e^{i\eta_1x},\hdots,e^{i\eta_nx})z).\]
It then follows that
\[p_{\mathbf{\eta}}\sigma(1)p_{\mathbf{\eta}}^{-1}:(r,x,z)\mapsto (r,x+1,\diag(e^{i(\eta_1+\theta_1)},\hdots,e^{i(\eta_n+\theta_n)})z).\]
This gives us a bundle isomorphism to a flat bundle over $C$ associated to a representation with finite image (since the $\eta_i+\theta_i$ are rational). Continuing as in the proof of Lemma 3.2 of~\cite{Mag15}, we may take $\eta\rightarrow 0$ while $\eta_i+\theta_i\in\mathbb{Q}$ and obtain the inequality in the general situation.
\end{proof}
Lemma~\ref{lowcusp} allows us to define the cuspidal resolvent
\[R_{\chi,\blambda(n)}(s):=(\Delta_{\chi,\blambda(n)}-s(1-s))^{-1}:L^2(C;E_{\chi,\blambda(n)}|_C)\rightarrow W^2(C;E_{\chi,\blambda(n)}|_C)\]
as a bounded holomorphic operator for $\Re(s)>\frac{1}{2}$, each of which is injective. Given a section $\omega\in L^2(C;E_{\chi,\blambda(n)}|_C)$, we have from
\[(\Delta_{\chi,\blambda(n)}-s(1-s))R_{\chi,\blambda(n)}(s)\omega=\omega\]
the inequality
\[\|R_{\chi,\blambda(n)}(s)\omega\|_{L^2(C;E_{\chi,\blambda(n)}|_C)}\leq\left(\frac{1}{4}-s(1-s)\right)^{-1}\|\omega\|_{L^2(C;E_{\chi,\blambda(n)}|_C)}.\]
Since
\[\Delta_{\chi,\blambda(n)}R_{\chi,\blambda(n)}(s)\omega=\omega+s(1-s)R_{\chi,\blambda(n)}(s)\omega,\]
we also obtain
\[\|\Delta_{\chi,\blambda(n)}R_{\chi,\blambda(n)}(s)\omega\|_{L^2(C;E_{\chi,\blambda(n)}|_C)}\leq \left(1+s(1-s)\left(\frac{1}{4}-s(1-s)\right)^{-1}\right)\|\omega\|_{L^2(C;E_{\chi,\blambda(n)}|_C)}=\frac{\|\omega\|_{L^2(C;E_{\chi,\blambda(n)}|_C)}}{1-4s(1-s)}.\]
We define the \textit{cusp parametrix} as
\[P^{\chi,\blambda(n)}_{\textup{cusp}}(s):=\chi^+_{C,\chi}R_{\chi,\blambda(n)}(s)\chi^-_{C,\chi}:L^2(X;E_{\chi,\blambda(n)})\rightarrow L^2(X;E_{\chi,\blambda(n)}).\]
As a composition of bounded operators, this is bounded. Note that
\begin{eqnarray*}
&&\chi^-_{C,\chi}+[\Delta_{\chi,\blambda(n)},\chi^+_{C,\chi}]R_{\chi,\blambda(n)}(s)\chi^-_{C,\chi}\\&=&\chi^-_{C,\chi}-\chi^+_{C,\chi}\Delta_{\chi,\blambda(n)}R_{\chi,\blambda(n)}(s)\chi^-_{C,\chi}+\Delta_{\chi,\blambda(n)}\chi^+_{C,\chi}R_{\chi,\blambda(n)}(s)\chi^-_{C,\chi}\\&=&
\chi^-_{C,\chi}-\chi^+_{C,\chi}(\chi^-_{C,\chi}+s(1-s)R_{\chi,\blambda(n)}\chi^-_{C,\chi})+\Delta_{\chi,\blambda(n)}\chi^+_{C,\chi}R_{\chi,\blambda(n)}(s)\chi^-_{C,\chi}\\&=&(\Delta_{\chi,\blambda(n)}-s(1-s))\chi^+_{C,\chi}R_{\chi,\blambda(n)}(s)\chi^-_{C,\chi}\\&=&(\Delta_{\chi,\blambda(n)}-s(1-s))P^{\chi,\blambda(n)}_{\textup{cusp}}(s).
\end{eqnarray*}
Letting $Q^{\chi,\blambda(n)}_{\textup{cusp}}(s):=[\Delta_{\chi,\blambda(n)},\chi^+_{C,\chi}]R_{\chi,\blambda(n)}(s)\chi^-_{C,\chi}$, we have from the above computation that
\begin{equation}\label{cuspequation}(\Delta_{\chi,\blambda(n)}-s(1-s))P^{\chi,\blambda(n)}_{\textup{cusp}}(s)=\chi^-_{C,\chi}+Q^{\chi,\blambda(n)}_{\textup{cusp}}(s).
\end{equation}
\begin{lemma}For each $s_0>\frac{1}{2}$, there is a constant $C(s_0)>0$ such that for every $s\in[s_0,1]$, $Q^{\chi,\blambda(n)}_{\textup{cusp}}(s)$ is a bounded operator such that
\[\|Q^{\chi,\blambda(n)}_{\textup{cusp}}(s)\|_{L^2}\leq C(s_0)\left(\|\Delta\chi^+_{C}\|_{\infty}+2\|\nabla\chi^+_{C}\|_{\infty}\right).\]
\end{lemma}
\begin{proof}
This is essentially Lemma 4.3 of Hide--Magee~\cite{HM} adapted to flat unitary bundles.
\end{proof}
As a corollary, we may deduce the following.
\begin{corollary}For any $s_0>\frac{1}{2}$, we can choose $\chi^+_C$ and $\chi^-_C$ depending on $s_0$, such that for any $s\in[s_0,1]$ and any $\chi:\Gamma\rightarrow U(n)$, we have
\[\|Q^{\chi,\blambda(n)}_{\textup{cusp}}(s)\|_{L^2}\leq\frac{1}{5}.\]
\end{corollary}
\subsection{Internal part}We now study the internal parametrix. We begin by discussing well-known results on the resolvent of the upper halfplane $\mathbb{H}$. Since $\Spec\Delta_{\mathbb{H}}\subset\left[\frac{1}{4},\infty\right)$, for each $s\in\mathbb{C}$ with $\Re(s)>\frac{1}{2}$, we have a \textit{bounded} operator
\[R_{\mathbb{H}}(s):=(\Delta_{\mathbb{H}}-s(1-s))^{-1}:L^2(\mathbb{H})\rightarrow L^2(\mathbb{H}).\]
For $x,y\in\mathbb{H}$, we let $r:=d(x,y)$ be the distance between $x$ and $y$ in the hyperbolic metric. Letting $\sigma:=\cosh^2(r/2)$, the operator $R_{\mathbb{H}}(s)$ is an integral operator with kernel
\[R_{\mathbb{H}}(s;x,y):=\frac{1}{4\pi}\int_0^1\frac{t^{s-1}(1-t)^{1-s}}{(\sigma-t)^s}dt.\]
Note that this depends only on $d(x,y)$, and so we conveniently let $R_{\mathbb{H}}(s;r):=R_{\mathbb{H}}(s;x,y)$. 
Taking a smooth function $\chi_0:\mathbb{R}\rightarrow [0,\infty)$ such that $\chi_0|_{(-\infty,0]}\equiv 1$ and $\chi_0|_{[1,\infty)}\equiv 0$, we define $\chi_T(t):=\chi_0(t-T)$ with support in $(-\infty,T+1]$ and let
\begin{equation}\label{RT}R^{(T)}_{\mathbb{H}}(s;r):=\chi_T(r)R_{\mathbb{H}}(s;r).\end{equation}
Letting
\[L^{(T)}_{\mathbb{H}}(s;r):=\left(-\frac{\partial^2\chi_T}{\partial^2r}-\frac{1}{\tanh r}\frac{\partial\chi_T}{\partial r}\right)R_{\mathbb{H}}(s;r)-2\frac{\partial\chi_T}{\partial r}\frac{\partial R_{\mathbb{H}}}{\partial r}(s;r),\]
we obtain the associated integral operator $L^{(T)}_{\mathbb{H}}(s)$. By Lemma 5.2 of Hide--Magee~\cite{HM}, there is a constant $C>0$ such that for any $T>0$ and $s\in[\frac{1}{2},1]$, $L^{(T)}_{\mathbb{H}}(s)$ extends to a bounded operator on $L^2(\mathbb{H})$ with operator norm
\[\|L^{(T)}_{\mathbb{H}}(s)\|_{L^2}\leq CTe^{(\frac{1}{2}-s)T}.\]
Therefore, for any $s_0>\frac{1}{2}$, there is a constant $T=T(s_0)$ such that for $s\in [s_0,1]$, we have
\begin{equation}\label{LTbound}\|L^{(T)}_{\mathbb{H}}(s)\|_{L^2}\leq\frac{1}{5}.\end{equation}
In the following $W^2(\mathbb{H})$ is the Sobolev space obtained by completing $C^{\infty}_c(\mathbb{H})$ with respect to the Sobolev norm
\[\|f\|^2_{W^2}=\|f\|^2_{L^2}+\|\Delta f\|^2_{L^2}.\]
Hide--Magee~\cite{HM} proved the following. 
\begin{lemma}For any $T>0$ and any $s\in (\frac{1}{2},1]$, given any compact $K\subset\mathbb{H}$, there is $C=C(s,K,T)>0$ such that
\begin{enumerate}
\item for any $f\in C^{\infty}_c(\mathbb{H})$ such that $\supp f\subset K$ we have $R^{(T)}_{\mathbb{H}}(s)f\in W^2(\mathbb{H})$  and 
\[\|R^{(T)}_{\mathbb{H}}(s)f\|_{W^2}\leq C\|f\|_{L^2};\]
\item with $f$ as above,
\[(\Delta_{\mathbb{H}}-s(1-s))R^{(T)}_{\mathbb{H}}f=f+L^{(T)}_{\mathbb{H}}f\]
as $L^2$ functions.
\end{enumerate}
\end{lemma}
We will now construct the internal parametrix for the compact part of our cusped hyperbolic surface. Let us define
\[R^{(T)}_{\blambda(n)}(s;x,y):=R^{(T)}_{\mathbb{H}}(s;x,y)\id_{V_{\blambda(n)}},\]
\[L^{(T)}_{\blambda(n)}(s;x,y):=L^{(T)}_{\mathbb{H}}(s;x,y)\id_{V_{\blambda(n)}},\]
giving us, respectively, the integral operators $R^{(T)}_{\blambda(n)}(s)$ and $L^{(T)}_{\blambda(n)}(s)$ on functions $\mathbb{H}\rightarrow V_{\blambda(n)}$. We need the following lemma, an analogue of Lemma 5.5 of~\cite{HM}. In the following, the subscript $\chi$ denotes that the automorphy condition is satisfied for the surface representation $\pi_{\blambda(n)}\circ\chi$, while the subscript $c$ denotes that the functions are compactly supported on $X$ when viewed as sections of the bundle $E_{\chi,\blambda(n)}$.
\begin{lemma}\label{essentialint}For all $s\in(\frac{1}{2},1]$,
\begin{enumerate}
\item the integral operator $R^{(T)}_{\blambda(n)}(s)(1-\chi^{-}_{C,\chi})$ is well-defined on $C^{\infty}_{c,\chi}(\mathbb{H},V_{\blambda(n)})$ and extends to a bounded operator
\[R^{(T)}_{\blambda(n)}(s)(1-\chi^{-}_{C,\chi}):L^2_{\chi}(\mathbb{H},V_{\blambda(n)})\rightarrow W^2_{\chi}(\mathbb{H},V_{\blambda(n)});\]
\item the integral operator $L^{(T)}_{\blambda(n)}(s)(1-\chi^{-}_{C,\chi})$ is well defined on $C^{\infty}_{c,\chi}(\mathbb{H},V_{\blambda(n)})$ and extends to a bounded operator on $L^2_{\chi}(\mathbb{H},V_{\blambda(n)})$; and
\item we have
\[(\Delta_{\mathbb{H}}-s(1-s))R^{(T)}_{\blambda(n)}(s)(1-\chi^{-}_{C,\chi})=(1-\chi^{-}_{C,\chi})+L^{(T)}_{\blambda(n)}(s)(1-\chi^{-}_{C,\chi})\]
as operators on $L^2_{\chi}(\mathbb{H},V_{\blambda(n)})$.
\end{enumerate}
\end{lemma}
Using this Lemma, we define the \textit{internal parametrix} as the operator
\[P^{\chi,\blambda(n)}_{\textup{int}}(s):=R^{(T)}_{\blambda(n)}(s)(1-\chi^{-}_{C,\chi}):L^2(X;E_{\chi,\blambda(n)})\rightarrow W^2(X;E_{\chi,\blambda(n)})\]
Additionally, defining
\[Q^{\chi,\blambda(n)}_{\textup{int}}(s):=L^{(T)}_{\blambda(n)}(s)(1-\chi^{-}_{C,\chi}):L^2(X;E_{\chi,\blambda(n)})\rightarrow L^2(X;E_{\chi,\blambda(n)}),\]
we obtain from Lemma~\ref{essentialint} the identity
\begin{equation}\label{intequation}(\Delta_{\mathbb{H}}-s(1-s))P^{\chi,\blambda(n)}_{\textup{int}}(s)=(1-\chi^{-}_{C,\chi})+Q^{\chi,\blambda(n)}_{\textup{int}}(s).\end{equation}
\subsection{Resolvent}By combining the results of the previous two sections, we define the bounded operator
\[P^{\chi,\blambda(n)}(s):=P^{\chi,\blambda(n)}_{\textup{int}}(s)+P^{\chi,\blambda(n)}_{\textup{cusp}}(s):L^2(X;E_{\chi,\blambda(n)})\rightarrow W^2(X;E_{\chi,\blambda(n)}).\]
From equations~\eqref{cuspequation} and~\eqref{intequation}, we obtain
\begin{equation}\label{operatorequation}(\Delta_{\chi,\blambda(n)}-s(1-s))P^{\chi,\blambda(n)}(s)=1+Q^{\chi,\blambda(n)}_{\textup{int}}(s)+Q^{\chi,\blambda(n)}_{\textup{cusp}}(s).\end{equation}
From~\eqref{LTbound} it follows that given $s>\frac{1}{2}$, we may suitably choose $T$ such that
\begin{equation}\label{Qcuspbound}\|Q^{\chi,\blambda(n)}_{\textup{cusp}}(s)\|_{L^2}\leq\frac{1}{5}.\end{equation}
Therefore, in order to show that $Q^{\chi,\blambda(n)}_{\textup{int}}(s)+Q^{\chi,\blambda(n)}_{\textup{cusp}}(s)$ is a bounded operator with norm strictly less than $1$ asymptotically almost surely over the space of surface representations $\chi$ as $n\rightarrow\infty$, it suffices to show this for $Q^{\chi,\blambda(n)}_{\textup{int}}(s)$. We prove this in the next subsection.
\subsection{Bounds on random operators}
Suppose $\omega\in C^{\infty}(X;E_{\chi,\blambda(n)})\simeq C^{\infty}_{\chi}(\mathbb{H};V_{\blambda(n)})$ is a section of the flat unitary bundle $E_{\chi,\blambda(n)}\rightarrow X$ such that $\|\omega\|_{L^2}<\infty$. By construction,
\begin{eqnarray*}Q^{\chi,\blambda(n)}_{\textup{int}}(s)[\omega](x)&=&\int_{\mathbb{H}}L^{(T)}_{\mathbb{H}}(s;x,y)(1-\chi^{-}_{C,\chi}(y))\omega(y)d\mu(y)\\&=& \sum_{\gamma\in\Gamma}\int_FL^{(T)}_{\mathbb{H}}(s;\gamma x,y)\pi_{\blambda(n)}(\chi(\gamma^{-1}))(1-\chi^{-}_{C,\chi}(y))\omega(y)d\mu(y),
\end{eqnarray*}
where $d\mu(y)$ is the hyperbolic area measure on $\mathbb{H}$.\\
\\
Let $F$ be a fundamental domain of $X$. From~\eqref{RT}, $L^{(T)}_{\mathbb{H}}(s;x,y)$ is nonzero only if $d(x,y)\leq T+1$. From this it follows that there is a compact subset $K(T)\subseteq F$ and a finite set $S=S(T)\subset \Gamma$ such that for $x,y\in \overline{F}$ and $\gamma\in\Gamma$,
\[L^{(T)}_{\mathbb{H}}(s;\gamma x,y)\pi_{\blambda(n)}(\chi(\gamma^{-1}))(1-\chi^{-}_{C,\chi}(y))=0\]
unless $x,y\in K$ and $\gamma\in S$.\\
\\
We have an isomorphism of Hilbert spaces
\[L^2(X;E_{\chi,\blambda(n)})\simeq L^2(F)\otimes V_{\blambda(n)}\]
\[\omega\mapsto \sum_i\left<\omega|_F,e_i\right>\otimes e_i,\]
where $\{e_i\}$ is an orthonormal basis of $V_{\blambda(n)}$. Under this isomorphism, we may write
\[Q^{\chi,\blambda(n)}_{\textup{int}}(s)=\sum_{\gamma\in S}a_{\gamma}^{(T)}(s)\otimes \pi_{\blambda(n)}(\chi(\gamma^{-1})),\]
where $a_{\gamma}^{(T)}(s):L^2(F)\rightarrow L^2(F)$ is given by
\[a_{\gamma}^{(T)}(s)[f](x):=\int_FL^{(T)}_{\mathbb{H}}(s;\gamma x,y)(1-\chi^{-}_{C}(y))f(y)d\mu(y).\]
$a_{\gamma}^{(T)}(s)$ are bounded operators on $L^2(F)$. We will approximate the operator $Q^{\chi,\blambda(n)}_{\textup{int}}(s)$ by the operator
\[L^{(T)}_{s,\infty}:=\sum_{\gamma\in S}a_{\gamma}^{(T)}(s)\otimes \pi_{\infty}(\gamma^{-1}),\]
where
\[\pi_{\infty}:\Gamma\rightarrow\textup{End}(\ell^2(\Gamma))\]
is the right regular representation. Under the isomorphism
\[L^2(F)\otimes\ell^2(\Gamma)\xrightarrow{\sim} L^2(\mathbb{H})\]
given by $f\otimes\delta_{\gamma}\mapsto f\circ\gamma^{-1}$, $L^{(T)}_{s,\infty}$ is conjugated to the \textit{bounded} operator
\[L^{(T)}_{\mathbb{H}}(s)(1-\chi^{-}_C):L^2(\mathbb{H})\rightarrow L^2(\mathbb{H}).\]
Corollary 6.2 of~\cite{HM} implies that for any $s_0>\frac{1}{2}$, there is a $T=T(s_0)>0$ such that for any $s\in[s_0,1]$,
\[\|L^{(T)}_{s,\infty}\|_{L^2(F)\otimes\ell^2(\Gamma)}\leq\frac{1}{5}.\]
We prove the following proposition.
\begin{proposition}
For any $s_0>\frac{1}{2}$, there is a $T=T(s_0)>0$ such that for any fixed $s\in[s_0,1]$, 
\[\mathbb{P}\left[\chi:\pi_1(X)\rightarrow U(n): \|Q^{\chi,\blambda(n)}_{\textup{int}}(s)\|_{L^2(F)\otimes V_{\blambda(n)}}\leq\frac{2}{5}\right]\xrightarrow{n\rightarrow\infty} 1.\]
\end{proposition}
\begin{proof}
By Corollary 6.2 of~\cite{HM} mentioned at the end of the previous subsection, there is a $T=T(s_0)>0$ such that for $s\in [s_0,1]$
\[\|L^{(T)}_{s,\infty}\|_{L^2(F)\otimes\ell^2(\Gamma)}\leq\frac{1}{5}.\]
For each $\gamma\in S=S(T)$ ($S$ is a finite set) there is a finite dimensional subspace $V_{\gamma}\subset L^2(F)$ and operator $b^{(T)}_{\gamma}(s):V_{\gamma}\rightarrow V_{\gamma}$ such that, after extending it by $0$ to an operator on $L^2(F)$, we have
\[\|a_{\gamma}^{(T)}(s)-b_{\gamma}^{(T)}(s)\|_{L^2(F)}\leq \frac{1}{20|S|}.\]
Since $S$ is finite, we may choose $V_{\gamma}$ uniformly by replacing them with $V=\sum V_{\gamma}\subset L^2(F)$. From this, we obtain that
\[\|Q^{\chi,\blambda(n)}_{\textup{int}}(s)-\sum_{\gamma\in S}b_{\gamma}^{(T)}(s)\otimes\pi_{\blambda(n)}(\chi(\gamma^{-1}))\|_{L^2(F)\otimes V_{\blambda(n)}}\leq \frac{1}{20}.\]
From Theorem~\ref{ThmBC}, with probability tending to $1$ as $n\rightarrow\infty$, surface representations $\chi:\Gamma\rightarrow U(n)$ satisfy the inequality
\[\|\sum_{\gamma\in S}b^{(T)}_{\gamma}(s)\otimes \pi_{\blambda(n)}(\chi(\gamma^{-1}))\|_{V\otimes V_{\blambda(n)}}\leq \|\sum_{\gamma\in S}b^{(T)}_{\gamma}(s)\otimes \pi_{\infty}(\gamma^{-1})\|_{V\otimes \ell^2(\Gamma)}+\frac{1}{20}.\]
On the other hand,
\begin{eqnarray*}
&&\|\sum_{\gamma\in S}b^{(T)}_{\gamma}(s)\otimes \pi_{\infty}(\gamma^{-1})\|_{V\otimes \ell^2(\Gamma)}\\&\leq&
\|L^{(T)}_{s,\infty}\|_{L^2(F)\otimes\ell^2(\Gamma)}+
\|L^{(T)}_{s,\infty}-\sum_{\gamma\in S}b^{(T)}_{\gamma}(s)\otimes \pi_{\infty}(\gamma^{-1})\|_{V\otimes \ell^2(\Gamma)}\\&\leq&
\frac{1}{5}+\|\sum_{\gamma\in S}(a^{(T)}_{\gamma}(s)-b^{(T)}_{\gamma}(s))\otimes \pi_{\infty}(\gamma^{-1})\|_{V\otimes \ell^2(\Gamma)}\\&\leq&\frac{1}{5}+\frac{1}{20}=\frac{1}{4}.
\end{eqnarray*}
It is then clear that combining the above inequalities, we obtain
\[\|Q^{\chi,\blambda(n)}_{\textup{int}}(s)\|_{L^2(F)\otimes V_{\blambda(n)}}\leq\frac{2}{5}\]
\end{proof}
\subsection{Proof of Theorem~\ref{Thm1}} 
From Lemma 6.1 of~\cite{HM}, we know that for each fixed $T>0$, there is a constant $c_1=c_1(T)>0$ such that for $s_1,s_2\in[\frac{1}{2},1]$ and $\gamma\in S(T)$ we have
\begin{equation}\label{uniformcont}\|a_{\gamma}^{(T)}(s_1)-a_{\gamma}^{(T)}(s_2)\|_{L^2(F)}\leq c_1|s_1-s_2|.\end{equation}
Given $\varepsilon>0$, let $s_0=\frac{1}{2}+\sqrt{\varepsilon}$ so that $s_0(1-s_0)=\frac{1}{4}-\varepsilon$. Choose $T=T(s_0)$ so that
\[\|Q^{\chi,\blambda(n)}_{\textup{cusp}}(s)\|_{L^2}\leq\frac{1}{5}\]
for every $s\in (s_0,1]$. \eqref{uniformcont} implies that
\begin{equation}\label{distance}\|Q^{\chi,\blambda(n)}_{\textup{int}}(s_1)-Q^{\chi,\blambda(n)}_{\textup{int}}(s_1)\|_{L^2(F)\otimes V_{\blambda(n)}}\leq |S|c_1|s_1-s_2|\end{equation}
for every $s_1,s_2\in[s_0,1]$. Choose a finite subset $X$ of $[s_0,1]$ such that every point of $[s_0,1]$ is within $\frac{1}{5|S|c_1}$ of a point of $X$. For each point $s\in X$,
\[\mathbb{P}\left[\chi:\pi_1(X)\rightarrow U(n):\|Q^{\chi,\blambda(n)}_{\textup{int}}(s)\|_{L^2}\leq\frac{2}{5}\right]\xrightarrow{n\rightarrow\infty}1.\]
From~\eqref{distance}, we obtain for every $s\in[s_0,1]$,
\[\mathbb{P}\left[\chi:\pi_1(X)\rightarrow U(n):\|Q^{\chi,\blambda(n)}_{\textup{int}}(s)\|_{L^2}\leq\frac{3}{5}\right]\xrightarrow{n\rightarrow\infty}1\]
From~\eqref{operatorequation},
\[(\Delta_{\chi,\blambda(n)}-s(1-s))P^{\chi,\blambda(n)}(s)=1+Q^{\chi,\blambda(n)}_{\textup{int}}(s)+Q^{\chi,\blambda(n)}_{\textup{cusp}}(s).\]
We have
\[\mathbb{P}\left[\chi:\pi_1(X)\rightarrow U(n):\|Q^{\chi,\blambda(n)}_{\textup{int}}(s)+Q^{\chi,\blambda(n)}_{\textup{cusp}}(s)\|_{L^2}\leq \frac{4}{5}\right]\xrightarrow{n\rightarrow\infty}1.\]
Therefore, for every $s\in[s_0,1]$,
\begin{equation}\label{invertible}\mathbb{P}\left[\chi:\pi_1(X)\rightarrow U(n):1+Q^{\chi,\blambda(n)}_{\textup{int}}(s)+Q^{\chi,\blambda(n)}_{\textup{cusp}}(s)\textup{ is invertible}\right]\xrightarrow{n\rightarrow\infty}1.\end{equation}
It then follows from~\eqref{operatorequation} that asymptotically almost surely we have
\begin{equation*}(\Delta_{\chi,\blambda(n)}-s(1-s))P^{\chi,\blambda(n)}(s)\left(1+Q^{\chi,\blambda(n)}_{\textup{int}}(s)+Q^{\chi,\blambda(n)}_{\textup{cusp}}(s)\right)^{-1}=1
.\end{equation*}
Therefore, for every $\varepsilon>0$, 
\[\mathbb{P}\left[\chi:\pi_1(X)\rightarrow U(n):\inf\Spec(\Delta_{\chi,\blambda(n)})\geq\frac{1}{4}-\varepsilon\right]\xrightarrow{n\rightarrow\infty}1.\]
This follows from~\eqref{invertible}.
\section{Random matrix theory estimates}\label{random}
There will be certain flat unitary rank $n$ bundles over our cusped hyperbolic surfaces that will not lead to the estimates that we need in this paper. However, the intent of this section is to prove results that will allow us to show that asymptotically almost surely as $n\rightarrow\infty$, our flat bundle is not one of these undesirable bundles. Therefore, we may discard them.
\subsection{Bounding eigenvalues away from $1$}
For a square matrix $A$, we write $\textup{Spec}(A)$ for the finite set of eigenvalues of $A$. Note that it is possible for a non-trivial irreducible representation $\pi: U(n)\rightarrow\GL(V)$ to be such that $\pi(A)$ has $1$ as an eigenvalue for every $A\in U(n)$. For example, one could take the representation $\det^{-1}\otimes\Sym^2$ of $U(2)$. However, for a large class of signatures $\blambda$, this does not happen, as shown in the following proposition. Recall Definition~\ref{unbalanced}.
\begin{proposition}
If $\blambda$ is unbalanced, then
\[\mathbb{P}\left[T\in U(n):\ 1\notin\Spec(\pi_{\blambda}(T))\right]=1.\]
\end{proposition}
\begin{proof}
Write $\blambda\in\mathbb{Z}^n$ as $m_1\omega_1+\hdots+m_n\omega_n$, where $\omega_i=(1,\hdots,1,0,\hdots,0)$ with exactly $i$ ones, and $m_i\in\mathbb{Z}_{\geq 0}$ for $1\leq i<n$ and $m_n\in\mathbb{Z}$. $V_{\blambda}$ is an irreducible subrepresentation of the representation $\rho_{\blambda}$ of $U(n)$ given by
\[(\mathbb{C}^n)^{\otimes m_1}\otimes (\wedge^2\mathbb{C}^n)^{\otimes m_2}\otimes\dots\otimes (\wedge^n\mathbb{C}^n)^{\otimes m_n}.\]
From this, it is easy to see that each eigenvalue of $\pi_{\blambda}(T)$ is of the form $e^{ih(\theta_1,\hdots,\theta_n)}$, where $h$ is a non-zero linear form and $e^{i\theta_j}$ are the eigenvalues of $T$. Since $\blambda$ is unbalanced, the sum of the entries of $\blambda$ is nonzero, that is, $m_1+2m_2+\hdots+nm_n\neq 0$, and so none of the $h$ are identitically zero as linear forms; otherwise, $m_1+2m_2+\hdots+nm_n=h(1,\hdots,1)=0$. Therefore, with probability $1$, none of the eigenvalues of $\pi_{\blambda}(T)$ are $1$, from which the conclusion follows.
\end{proof}
Recall the notation from Subsection~\ref{spectraltheory}. We have the following corollary.
\begin{corollary}\label{cor:genericbundle}
If $\blambda$ is unbalanced, then
\[\mathbb{P}\left[\chi:\forall i\ 1\notin\Spec(\pi_{\blambda}(\chi(T_i)))\right]=1.\]
\end{corollary}
This implies that when the signature $\blambda$ of $U(n)$ is unbalanced, then with probability $1$ over the space of surface representations $\chi:\Gamma\rightarrow U(n)$, surface representations $\pi_{\blambda}\circ\chi$ lead to bundles $E_{\chi,\blambda}$ whose spectra have no continuous part. Consequently, the analysis of the hyperbolic scattering determinants of the flat bundles is only relevant for families of irreducible representations associated to balanced $\blambda$.\\
\\
We will also need the following proposition, even when $\blambda$ is unbalanced.
\begin{proposition}\label{logbound}Suppose $\blambda(n)$ is a sequence of signatures such that $|\blambda(n)|=O(\log n)$. Then, for any $\delta>0$,
\[\mathbb{P}\left[A\in U(n):\forall\alpha\in\textup{Spec}(\pi_{\blambda(n)}(A))\setminus\{1\},\ |1-\alpha|\geq\frac{1}{(2+\delta)^{n^2}}\right]\xrightarrow{n\rightarrow\infty}1.\]
\end{proposition}
\begin{proof}
We prove this proposition by first viewing the representation $\pi_{\blambda(n)}$ as a subrepresentation of another representation. Again, writing $\blambda(n)=m_1\omega_1+\hdots+m_n\omega_n$ as in the previous proposition, $\pi_{\blambda(n)}$ is an irreducible subrepresentation of $\rho_{\blambda(n)}$ given by
\[(\mathbb{C}^n)^{\otimes m_1}\otimes (\wedge^2\mathbb{C}^n)^{\otimes m_2}\otimes\dots\otimes (\wedge^n\mathbb{C}^n)^{\otimes m_n}.\]
Clearly, it suffices to prove the proposition for this larger representation which has the advantage of being explicit. Suppose the eigenvalues of $A$ are $e^{i\theta_1},\hdots,e^{i\theta_n}$. Then the eigenvalues of $\rho_{\blambda(n)}(A)$ are
\[e^{ih(\theta_1,\dots,\theta_n)},\]
where $h$ is a linear form in the $\theta_1,\hdots,\theta_n$. $h$ could be identically $0$, in which case the corresponding eigenvalue will be $1$. We may therefore consider only $h\neq 0$. Then for $f(n)$ a function of $n$ such that $f(n)\rightarrow+\infty$ as $n\rightarrow\infty$,
\begin{eqnarray*}&&\sum_{h\neq 0}\mathbb{P}\left[A\in U(n):|1-e^{ih(\theta_1,\dots,\theta_n)}|<\frac{1}{f(n)}\right]\\&=&\int_{U(n)}\sum_{h\neq 0}\chi_{|1-e^{ih(\theta_1,\dots,\theta_n)}|<\frac{1}{f(n)}}(A)d\mu(A)\\&=&
\frac{1}{n!}\int_{[0,2\pi]^n}\sum_{h\neq 0}\chi_{|1-e^{ih(\theta_1,\dots,\theta_n)}|<\frac{1}{f(n)}}(\diag(e^{i\theta_1},\hdots,e^{i\theta_n}))\left|\prod_{a<b}(e^{i\theta_a}-e^{i\theta_b})\right|^2\frac{d\theta_1}{2\pi}\dots \frac{d\theta_n}{2\pi}\\&\leq&
\frac{4^{\binom{n}{2}}}{n!}\int_{[0,2\pi]^n}\sum_{h\neq 0}\chi_{|1-e^{ih(\theta_1,\dots,\theta_n)}|<\frac{1}{f(n)}}(\diag(e^{i\theta_1},\hdots,e^{i\theta_n}))\frac{d\theta_1}{2\pi}\dots \frac{d\theta_n}{2\pi}\\&\leq&
\frac{4^{\binom{n}{2}}\dim(\rho_{\blambda(n)})}{n!}\max_{h\neq 0}\vol_{\mathbb{T}^n}\left(\left|\sin\left(\frac{h(\theta_1,\dots,\theta_n)}{2}\right)\right|<\frac{1}{2f(n)}\right)\\&\leq& 
\frac{4^{\binom{n}{2}}\dim(\rho_{\blambda(n)})}{n!}\max_{h\neq 0}\vol_{\mathbb{T}^n}\left(\left|h(\theta_1,\dots,\theta_n)\right|<\frac{1}{f(n)}\right)\\&=& O\left(\frac{4^{\binom{n}{2}}\prod_{k=1}^n\binom{n}{k}^{|m_k|}}{n!f(n)}\right).
\end{eqnarray*}
The second equality follows from the Weyl integration formula~\eqref{Weylintegration} for the class function $\sum_{h\neq 0}\chi_{|1-e^{ih(\theta_1,\dots,\theta_n)}|<\frac{1}{f(n)}}$ on $U(n)$. $\frac{d\theta_1}{2\pi}\dots \frac{d\theta_n}{2\pi}$ is the probability Haar measure on the torus $\mathbb{T}^n$. We have the following inequality
\[n!>\sqrt{2\pi n}\left(\frac{n}{e}\right)^ne^{\frac{1}{12n+1}}.\]
Therefore, by the arithmetic mean-geometric mean inequality, we have
\[\frac{4^{\binom{n}{2}}\prod_{k=1}^n\binom{n}{k}^{|m_k|}}{n!}\leq \frac{2^{n^2-n}2^{n(m_1+\hdots+|m_n|)}}{(m_1+\hdots+|m_n|)^{(m_1+\hdots+|m_n|)}\sqrt{2\pi n}\left(\frac{n}{e}\right)^ne^{\frac{1}{12n+1}}}.\]
Note that $m_1+\hdots+|m_n|\leq m_1+2m_2+\hdots+n|m_n|+(|m_n|-m_n)=|\blambda(n)|+2|\lambda_n(n)|=O(\log n)$. From this, it is easy to see that if $f(n)\geq (2+\delta)^{n^2}$ for some $\delta>0$, then
\[\frac{4^{\binom{n}{2}}\prod_{k=1}^n\binom{n}{k}^{|m_k|}}{n!f(n)}\xrightarrow{n\rightarrow\infty} 0.\]
The conclusion follows.
\end{proof}
\begin{remark}In this proposition, our choice of $f(n)\geq (2+\delta)^{n^2}$ is crude; however, this proposition is sufficient for our purposes.
\end{remark}

\section{Counting eigenvalues}\label{counting}
In this section, we prove a few probabilistic Weyl law type bounds counting eigenvalues of the Laplacian acting on sections of $E_{\chi,\blambda(n)}$.
\begin{proposition}\label{Weylbound}Given unbalanced signatures $\blambda(n)$ with $|\blambda(n)|=O(\log n)$, asymptotically almost surely as $n\rightarrow\infty$, surface representations $\chi:\Gamma\rightarrow U(n)$ lead to bundles $E_{\chi,\blambda(n)}$ whose eigenvalues satisfy
\[\sum_{0\leq r_k\leq T}1=O_X\left(n^2\dim V_{\blambda(n)}T^2\right)\]
for $T\geq 1$.
\end{proposition}
For the notation, see Subsection~\ref{spectraltheory}.
\begin{proof}
We prove this proposition using Selberg's trace formula for flat unitary rank $\dim(V_{\blambda(n)})$ bundles over finite area hyperbolic surfaces, which is the case with which this paper is concerned. By Corollary~\ref{cor:genericbundle}, we may assume without loss of generality that $K_0=0$, that is, our bundles are non-singular. Applying Theorem~\ref{Selbergtraceformula} to the bundle $E_{\chi,\blambda(n)}$ with the functions $h(r)=e^{-Tr^2}$ and $g(u)=\frac{e^{-\frac{u^2}{4T}}}{\sqrt{4\pi T}}$, we obtain
\begin{eqnarray*}
\sum_ke^{-Tr_k^2}&=& \frac{\dim V_{\blambda(n)}\area(X)}{4\pi}\int_{-\infty}^{\infty}re^{-Tr^2}\tanh(\pi r)dr\\&+&
\sum_{\tr(P)>2}\frac{\tr(\pi_{\blambda(n)}(\chi(P)))\log N(P_0)}{N(P)^{1/2}-N(P)^{-1/2}}\frac{e^{-\frac{(\log N(P))^2}{4T}}}{\sqrt{4\pi T}}\\&-&
\frac{1}{\sqrt{4\pi T}}\sum_{\alpha_{hj}}\log|1-e^{2\pi i\alpha_{hj}}|,
\end{eqnarray*}
where the $e^{2\pi i\alpha_{hj}}$ are the eigenvalues of $\pi_{\blambda(n)}(\chi(T_j))$, $h\in\{1,\hdots,n\}$. Recall the notation from Subsection~\ref{spectraltheory} where $T_j$ are parabolic elements associated to the cusps of $X$.\\
\\
We have the following bounds which follow from simple variations of the usual arguments that go into proving Weyl laws. We have
\begin{equation}\label{geobound}
\sum_{\tr(P)>2}\frac{\tr(\pi_{\blambda(n)}(\chi(P)))\log N(P_0)}{N(P)^{1/2}-N(P)^{-1/2}}\frac{e^{-\frac{(\log N(P))^2}{4T}}}{\sqrt{4\pi T}}=O_X(\dim V_{\blambda(n)}),
\end{equation}
and
\begin{equation}\label{trivialbound}
\frac{\dim V_{\blambda(n)}\area(X)}{4\pi}\int_{-\infty}^{\infty}re^{-Tr^2}\tanh(\pi r)dr=\frac{\dim V_{\blambda(n)}\area(X)}{4\pi T}+O_X(\dim V_{\blambda(n)}).
\end{equation}
By assumption, $|\blambda(n)|=O(\log n)$. By Proposition~\ref{logbound}, asymptotically almost surely as $n\rightarrow\infty$,
\[\sum_{\alpha_{hj}}\log|1-e^{2\pi i\alpha_{hj}}|=O_X(n^2\dim V_{\blambda(n)}).\]
Consequently,
\begin{equation}\label{logboundeq}\frac{1}{\sqrt{4\pi T}}\sum_{\alpha_{hj}}\log|1-e^{2\pi i\alpha_{hj}}|=O_X\left(\frac{n^2\dim V_{\blambda(n)}}{\sqrt{T}}\right).
\end{equation}
Combining~\eqref{geobound},~\eqref{trivialbound}, and~\eqref{logboundeq} together, we obtain that asymptotically almost surely as $n\rightarrow\infty$, bundles have the property that
\begin{equation}\label{finalbound}\sum_{r_k}e^{-Tr_k^2}\leq O_X\left(\frac{n^2\dim V_{\blambda(n)}}{\sqrt{T}}\right)+O_X\left(\frac{\dim V_{\blambda(n)}}{T}\right).
\end{equation}
Note that the number of eigenvalues of $E_{\chi,\blambda(n)}$ less than $1/4$, corresponding to $r_k\notin\mathbb{R}$, is $O_X(\dim V_{\blambda(n)})$. Therefore,
\[\sum_{r_k\notin\mathbb{R}}1=O_X(\dim V_{\blambda(n)}).\]
Replacing $T$ with $1/T^2$, and putting the above together, we obtain that asymptotically almost surely as $n\rightarrow\infty$, 
\[\sum_{0\leq r_k\leq T}1\leq e\sum_{r_k\leq T}e^{-\frac{r_k^2}{T^2}}=O_X(n^2\dim V_{\blambda(n)}T^2)\]
for $T\geq 1$. The conclusion follows.
\end{proof}
\begin{proposition}\label{diffeigen}Suppose $\blambda(n)$ are unbalanced signatures with $|\blambda(n)|=O(\log n)$. Asymptotically almost surely as $n\rightarrow\infty$,
\begin{equation}
\sum_{|r_k-T|\leq 1}1=O_X(\dim V_{\blambda(n)}(n^2+n^2\log T+T))
\end{equation}
for $T\geq 1$.
\end{proposition}
\begin{proof}By Corollary~\ref{cor:genericbundle}, we may assume without loss of generality that $E_{\chi,\blambda(n)}$ is non-singular, that is $K_0=0$. Let $2T\leq \beta\leq 3T$ and write $\alpha=s-\frac{1}{2}=x+iT$, where $\frac{3}{4}\leq x\leq \frac{3}{2}$. Applying equation (2.8) of~\cite[p.435]{Hejhal2}, we see that $D_{\chi,\blambda(n)}$, the logarithmic derivative of Selberg's zeta function for $E_{\chi,\blambda(n)}$, satisfies the equation
\begin{align}
\sum_k\left(\frac{1}{r_k^2+(s-\frac{1}{2})^2}-\frac{1}{r_k^2+\beta^2}\right)&=\frac{\dim V_{\blambda(n)}\area(X)}{2\pi}\sum_{k=0}^{\infty}\left(\frac{1}{k+s-\frac{1}{2}}-\frac{1}{k+\beta}\right)\\&+\frac{D_{\chi,\blambda(n)}(s)}{2s-1}-\frac{D_{\chi,\blambda(n)}(\frac{1}{2}+\beta)}{2\beta}+\left(\frac{1}{2\beta}-\frac{1}{2s-1}\right)\sum_{\alpha_{hj}}\log|1-e^{2\pi i\alpha_{hj}}|.\nonumber
\end{align}
Since $|\blambda(n)|=O(\log n)$ and $\blambda(n)$ is unbalanced, Proposition~\ref{logbound} implies that asymptotically almost surely as $n\rightarrow\infty$
\[\sum_{\alpha_{hj}}\log|1-e^{2\pi i\alpha_{hj}}|=O_X(n^2\dim V_{\blambda(n)}).\]
Furthermore,
\begin{eqnarray*}
&&\frac{\dim V_{\blambda(n)}\area(X)}{2\pi}\sum_{k=0}^{\infty}\left(\frac{1}{k+s-\frac{1}{2}}-\frac{1}{k+\beta}\right)\\&=&\frac{\dim V_{\blambda(n)}\area(X)}{2\pi}\sum_{k=0}^{T}\left(\frac{\beta-s+\frac{1}{2}}{(k+s-\frac{1}{2})(k+\beta)}\right)+\frac{\dim V_{\blambda(n)}\area(X)}{2\pi}\sum_{k=T}^{\infty}\left(\frac{\beta-s+\frac{1}{2}}{(k+s-\frac{1}{2})(k+\beta)}\right)\\&=&O_X(\dim V_{\blambda(n)}T)\sum_{k=1}^T\frac{1}{T^2}+O_X(\dim V_{\blambda(n)}T)\sum_{k=T}^{\infty}\frac{1}{k^2}\\&=&O_X(\dim V_{\blambda(n)}).
\end{eqnarray*}
We obtain that asymptotically almost surely as $n\rightarrow\infty$,
\begin{equation}\label{Dsest}
D_{\chi,\blambda(n)}(s)=O_X(\dim V_{\blambda(n)}T)+O_X(n^2\dim V_{\blambda(n)})+\sum_k\left(\frac{2s-1}{r_k^2+(s-\frac{1}{2})^2}-\frac{2s-1}{r_k^2+\beta^2}\right).
\end{equation}
The number of eigenvalues of less than $1/4$, corresponding to $r_k\notin\mathbb{R}$, is at most $O_X(\dim V_{\blambda(n)})$, from which it follows that
\[\sum_{r_k\notin\mathbb{R}}\left(\frac{2s-1}{r_k^2+(s-\frac{1}{2})^2}-\frac{2s-1}{r^2+\beta^2}\right)=O_X(\dim V_{\blambda(n)}).\]
Note that $\Re(s)>1$, and so
\[D_{\chi,\blambda(n)}(s)=\sum_{P\in\Gamma_{\textup{hyp}}}\frac{\Lambda(P)\tr(\pi_{\blambda(n)}(\chi(P))}{N(P)^s}=O_X(\dim V_{\blambda(n)}).\]
Taking real parts in~\eqref{Dsest}, we obtain
\begin{equation}\label{Dsestreal}
O_X(\dim V_{\blambda(n)}(n^2+T))=\sum_{r_k\geq 0}\left(\frac{x}{x^2+(T-r_k)^2}+\frac{x}{x^2+(T+r_k)^2}-\frac{2x}{r_k^2+\beta^2}\right).
\end{equation}
Furthermore, note that
\begin{eqnarray*}
&&\sum_{r_k\geq 2T}\left|\frac{x}{x^2+(T-r_k)^2}+\frac{x}{x^2+(T+r_k)^2}-\frac{2x}{r_k^2+\beta^2}\right|\\&\leq&
\sum_{r_k\geq 2T}\left|\frac{x}{x^2+(T-r_k)^2}-\frac{x}{r_k^2+\beta^2}\right|+\sum_{0\leq r_k\leq 2T}\left|\frac{x}{x^2+(T+r_k)^2}-\frac{x}{r_k^2+\beta^2}\right|.
\end{eqnarray*} 
We bound the latter two sums. We show that for $\beta\geq 2T$, they are $O_X(n^2\dim V_{\blambda(n)}(1+\log T))$. In the following, $N(a\leq r_k<b)$ is the number of $r_k$ in $[a,b)$. We have
\begin{eqnarray*}
&&\sum_{r_k\geq 2T}\left|\frac{x}{x^2+(T-r_k)^2}-\frac{x}{r_k^2+\beta^2}\right|\\&=&
\sum_{j=2T}^{\infty}\sum_{j\leq r_k<j+1}\left|\frac{x}{x^2+(T-r_k)^2}-\frac{x}{r_k^2+\beta^2}\right|\\&\leq&
O(1)\sum_{j=2T}^{\infty}\sum_{j\leq r_k<j+1}\frac{1}{j^2}\\&=&
O(1)\sum_{j=2T}^{\infty}N(j\leq r_k<j+1)\frac{1}{j^2}\\&=&
O(1)\left(\lim_{S\rightarrow\infty}\frac{N(0\leq r_k<S)}{S^2}-\frac{N(0\leq r_k<2T)}{4T^2}-\sum_{j=2T+1}^{\infty}N(0\leq r_k<j)\left(\frac{1}{j^2}-\frac{1}{(j+1)^2}\right)\right)\\&=&
O_X(n^2\dim V_{\blambda(n)}(1+\log T)).
\end{eqnarray*}
In the second to last equality, summation by parts is used, while in the last equality, we used Proposition~\ref{Weylbound} giving $N(0\leq r_k\leq T)=O_X(n^2\dim V_{\blambda(n)}T^2)$. Similarly, we obtain the same bound on the second sum. Equation~\eqref{Dsestreal} becomes
\begin{equation}\label{Dsestreal2}
\sum_{0\leq r_k\leq 2T}\frac{x}{x^2+(T-r_k)^2}=O_X(\dim V_{\blambda(n)}(n^2+n^2\log T+T))
\end{equation}
for $T\geq 1$. Letting $x=1$, we obtain for $T\geq 1$ the bound
\[\sum_{|T-r_k|\leq 1}1\leq 2\sum_{0\leq r_k\leq 2T}\frac{1}{1+(T-r_k)^2}=O_X(\dim V_{\blambda(n)}(n^2+n^2\log T+T)),\]
as required.
\end{proof}
\section{Probabilistic prime geodesic theorem}\label{pgt}
In this section, we prove that the error term of the prime geodesic theorem for flat unitary bundles is asymptotically almost surely as $n\rightarrow\infty$ bounded in terms of $X$ and the signatures $\blambda(n)$ alone. By Corollary~\ref{cor:genericbundle}, almost surely the spectra of the flat bundles do not have continuous parts when the signatures $\blambda(n)$ are unbalanced. In these cases, the hyperbolic scattering determinant makes no appearance. At the end of this section, we combine the results of this section with Theorem~\ref{Thm1} to obtain Theorem~\ref{geodesicdep}.\\
\\
Our setup is as before. We have a family of unbalanced signatures $\blambda(n)$ with $|\blambda(n)|\leq c\frac{\log n}{\log\log n}$, where $c>0$ is the universal constant appearing in Theorem~\ref{ThmBC}. For such families of signatures, we will prove and use a prime geodesic theorem for rank $\dim V_{\blambda(n)}$ flat unitary bundles on a cusped hyperbolic surface $X$ obtained from the surface representations
\[\pi_1(X)\xrightarrow{\chi} U(n)\xrightarrow{\pi_{\blambda(n)}}\GL(V_{\blambda(n)}).\]
In the usual prime geodesic theorem, the error term depends on $\Gamma=\pi_1(X)$, the bundle, and the lengths of the geodesics. We prove a probabilistic prime geodesic theorem for flat unitary bundles where the error term only depends on $X$, the lengths of the geodesics, and the signatures $\blambda(n)$. Though the proof is roughly the same as the usual proof of the prime geodesic theory for flat unitary vector bundles, we keep track of the error terms in a probabilistic manner using the results we proved in previous sections.\\
\\
Recall that the Selberg zeta function associated to $\pi_{\blambda(n)}\circ\chi:\Gamma\rightarrow \GL(V_{\blambda(n)})$ is given by
\[Z_{\chi,\blambda(n)}(s):=\prod_{P_0\in P\Gamma_{\textup{hyp}}}\prod_{k=0}^{\infty}\det(I-\pi_{\blambda(n)}(\chi(P_0))N(P_0)^{-k-s}),\]
where $N(P_0)$ is the norm of $P_0$. The length of the closed geodesic corresponding to $P_0$ is $\log N(P_0)$. The Selberg zeta function is absolutely convergent for $\Re(s)>1$, and its logarithmic derivative, a Dirichlet series, is given by
\[D_{\chi,\blambda(n)}(s):=\frac{Z'_{\chi,\blambda(n)}(s)}{Z_{\chi,\blambda(n)}(s)}=\sum_{P\in\Gamma_{\textup{hyp}}}\frac{\Lambda(P)}{N(P)^s}\tr(\pi_{\blambda(n)}(\chi(P))).\]
Here,
\[\Lambda(P):=\frac{\log N(P_0)}{1-N(P)^{-1}},\]
where $P_0$ is the primitive conjugacy class associated to $P$. We will also use the function
\[\Psi^{\chi,\blambda(n)}(x):=\sum_{\substack{P\in\Gamma_{\textup{hyp}}\\N(P)\leq x}}\Lambda(P)\tr(\pi_{\blambda(n)}(\chi(P))).\]
for $x\geq 1$,
as well as
\[\Psi_1^{\chi,\blambda(n)}(x):=\int_1^x\Psi^{\chi,\blambda(n)}(t)dt.\]
We will need the following proposition.
\begin{proposition}\label{contourshift}For $x\geq 1$,
\[\Psi_1^{\chi,\blambda(n)}(x)=\frac{1}{2\pi i}\int_{2-i\infty}^{2+i\infty}\frac{x^{s+1}}{s(s+1)}D_{\chi,\blambda(n)}(s)ds.\]
\end{proposition}
\begin{proof}
This is standard and the argument is as in~\cite[p.31, Theorem B]{Ingham}.
\end{proof}

We integrate over the rectangular contour, oriented counterclockwise, that is the boundary of
\[R(A,T(n)):=[-A,2]\times [-T(n),T(n)]\]
for a specific $A$ that we choose later and $T(n)\geq 1$ a function of $n$ such that $T(n)\rightarrow\infty$. Note that in Theorem~\ref{geodesicdep}, we may assume without loss of generality that $T(n)\geq \left(n^2\dim V_{\blambda(n)}\right)^{\frac{1}{\frac{1}{4}-\varepsilon}}$. The theorem is trivially true otherwise. We assume the weaker lower bound $T(n)\geq n^3$. By Proposition~\ref{diffeigen}, we have that for $T\geq 1$
\[\sum_{|r_k-T|\leq 1}1=O_X(\dim V_{\blambda(n)}(n^2+n^2\log T+T)).\]
The above bound allows us to assume without loss of generality that
\begin{equation}\label{admissible}|r_k-T(n)|\geq \frac{C_X}{\dim V_{\blambda(n)}T(n)}\end{equation}
for every $k$, where $C_X>0$ is a constant depending only on $X$. We also assume throughout this paper that $A\in\frac{1}{2}+\mathbb{Z}$, and that $x\geq 2$.\\
\\
By applying Cauchy's residue theorem on the boundary $\partial R(A,T(n))$ oriented counterclockwise, we write
\begin{align}\label{Cauchyresidue}\frac{1}{2\pi i}\int_{2-iT(n)}^{2+iT(n)}\frac{x^s}{s(s+1)}D_{\chi,\blambda(n)}(s)ds&= \frac{1}{2\pi i}\int_{-A+iT(n)}^{-A-iT(n)}\frac{x^s}{s(s+1)}D_{\chi,\blambda(n)}(s)ds\\&+\frac{1}{2\pi i}\int_{-A+iT(n)}^{2+iT(n)}\frac{x^s}{s(s+1)}D_{\chi,\blambda(n)}(s)ds \nonumber \\&-\frac{1}{2\pi i}\int_{-A+iT(n)}^{2-iT(n)}\frac{x^s}{s(s+1)}D_{\chi,\blambda(n)}(s)ds \nonumber \\&+\sum_{R(A,T(n))}\textup{Res}\left(\frac{x^s}{s(s+1)}D_{\chi,\blambda(n)}\right),\nonumber
\end{align}
where the last summation is over all the residues of $\frac{x^s}{s(s+1)}D_{\chi,\blambda(n)}$ in $R(A,T(n))$. We proceed to estimate the various integrals. 
\begin{proposition}\label{leftvert}Asymptotically almost surely as $n\rightarrow\infty$, $\chi:\Gamma\rightarrow U(n)$ satisfy
\[\int_{-A-iT(n)}^{-A+iT(n)}\frac{x^{s+1}}{s(s+1)}D_{\chi,\blambda(n)}(s)ds=O_X\left(\left(1+\frac{n^2}{A}\right)\frac{\dim V_{\blambda(n)}x^{1-A}}{A\log x}\right).\]
\end{proposition}
\begin{proof}
Using the functional equation~\cite[Thm. 2.18, p.441]{Hejhal2} for the Dirichlet series $D_{\chi,\blambda(n)}$, we have
\begin{eqnarray*}
\int_{-A-iT(n)}^{-A+iT(n)}\frac{x^{s+1}}{s(s+1)}D_{\chi,\blambda(n)}(s)ds&=& -\int_{-A-iT(n)}^{-A+iT(n)}\frac{x^{s+1}}{s(s+1)}D_{\chi,\blambda(n)}(1-s)ds\\&+&\dim V_{\blambda(n)}\area(X)\int_{-A-iT(n)}^{-A+iT(n)}\frac{x^{s+1}}{s(s+1)}(s-\frac{1}{2})\cot(\pi s)ds\\&+&2\sum_{\alpha_{hj}}\log|1-e^{2\pi i\alpha_{hj}}|\int_{-A-iT(n)}^{-A+iT(n)}\frac{x^{s+1}}{s(s+1)}ds.
\end{eqnarray*}
Note that $|D_{\chi,\blambda(n)}(s)|\leq\dim V_{\blambda(n)}|D_{1,0}(s)|\leq O_X(\dim V_{\blambda(n)})$ for $\Re(s)>1$. Here, $D_{1,0}(s)$ is the Dirichlet series associated to the surface $X$ itself, that is, the trivial rank $1$ bundle over $X$. Furthermore, since we have chosen $A\in\frac{1}{2}+\mathbb{Z}$, $\cot(\pi(-A+it))=i\tanh(\pi t)$ is bounded above by $1$ in absolute value. Using Proposition~\ref{logbound}, the bound follows.
\end{proof}
\begin{proposition}
Asymptotically almost surely as $n\rightarrow\infty$, $\chi$ satisfy
\begin{equation}\label{tophor}
\int_{-A+iT(n)}^{2+iT(n)}\left|\frac{x^{s+1}}{s(s+1)}\right||D_{\chi,\blambda(n)}(s)||ds|\leq 2\int_{\frac{1}{2}+iT(n)}^{2+iT(n)}\frac{x^{\sigma+1}}{T(n)^2}|D_{\chi,\blambda(n)}(s)||ds|+O_X\left(\left(1+\frac{n^2}{T(n)}\right)\frac{x^{3/2}\dim V_{\blambda(n)}}{T(n)}\right),
\end{equation}
where $\sigma:=\Re(s)$.
\end{proposition}
\begin{proof}
We split the integral into three parts as follows.
\[\int_{-A+iT(n)}^{2+iT(n)}=\int_{-A+iT(n)}^{-1+iT(n)}+\int_{-1+iT(n)}^{\frac{1}{2}+iT(n)}+\int_{\frac{1}{2}+iT(n)}^{2+iT(n)}.\]
It is clear from the triangle inequality that
\begin{equation}\label{finalint}
\int_{\frac{1}{2}+iT(n)}^{2+iT(n)}\left|\frac{x^{s+1}}{s(s+1)}\right||D_{\chi,\blambda(n)}(s)||ds|\leq\int_{\frac{1}{2}+iT(n)}^{2+iT(n)}\frac{x^{\sigma+1}}{T(n)^2}|D_{\chi,\blambda(n)}(s)||ds|.
\end{equation}
Using the functional equation for $D_{\chi,\blambda(n)}$, for the second integral we obtain that asymptotically almost surely as $n\rightarrow\infty$,
\begin{eqnarray*}
&&\int_{-1+iT(n)}^{\frac{1}{2}+iT(n)}\left|\frac{x^{s+1}}{s(s+1)}\right||D_{\chi,\blambda(n)}(s)||ds|\\&\leq&
\frac{x^{3/2}}{T(n)^2}\int_{-1+iT(n)}^{\frac{1}{2}+iT(n)}\left(|D_{\chi,\blambda(n)}(1-s)|+\dim V_{\blambda(n)}\area(X)\left|s-\frac{1}{2}\right||\cot(\pi s)|+2\left|\sum_{\alpha_{hj}}\log|1-e^{2\pi i\alpha_{hj}}|\right|\right)|ds|\\&\leq&
\frac{x^{3/2}}{T(n)^2}\int_{\frac{1}{2}+iT(n)}^{2+iT(n)}|D_{\chi,\blambda(n)}(s)||ds|+O_X\left(\frac{\dim V_{\blambda(n)}x^{3/2}}{T(n)}\right)+O_X\left(\frac{n^2\dim V_{\blambda(n)}x^{3/2}}{T(n)^2}\right)\\&\leq&
\int_{\frac{1}{2}+iT(n)}^{2+iT(n)}\frac{x^{\sigma+1}}{T(n)^2}|D_{\chi,\blambda(n)}(s)||ds|+O_X\left(\left(1+\frac{n^2}{T(n)}\right)\frac{x^{3/2}\dim V_{\blambda(n)}}{T(n)}\right).
\end{eqnarray*}
Here, we have used that
\[|\cot(\pi s)|\leq 2\textup{ for } \Im(s)\geq 1.\]
The bound on the first integral follows in a similar fashion. Indeed, we have asymptotically almost surely as $n\rightarrow\infty$,
\begin{eqnarray*}
&&\int_{-A+iT(n)}^{-1+iT(n)}\left|\frac{x^{s+1}}{s(s+1)}\right||D_{\chi,\blambda(n)}(s)||ds|\\&\leq&
\frac{x^{3/2}}{T(n)^2}\int_{-A+iT(n)}^{-1+iT(n)}\left(|D_{\chi,\blambda(n)}(1-s)|+\dim V_{\blambda(n)}\area(X)\left|s-\frac{1}{2}\right||\cot(\pi s)|+2\left|\sum_{\alpha_{hj}}\log|1-e^{2\pi i\alpha_{hj}}|\right|\right)|ds|\\&\leq&
\frac{x^{3/2}}{T(n)^2}\int_{1+A-iT(n)}^{2-iT(n)}|D_{\chi,\blambda(n)}(s)||ds|+O_X\left(\frac{\dim V_{\blambda(n)}x^{3/2}}{T(n)}\right)+O_X\left(\frac{n^2\dim V_{\blambda(n)}x^{3/2}}{T(n)^2}\right)\\&\leq&
\frac{x^{3/2}}{T(n)^2}|D_{\chi,\blambda(n)}(2)|+O_X\left(\left(1+\frac{n^2}{T(n)}\right)\frac{x^{3/2}\dim V_{\blambda(n)}}{T(n)}\right).
\end{eqnarray*}
The conclusion follows by noting that $|D_{\chi,\blambda(n)}(2)|\leq\dim V_{\blambda(n)}|D_{1,0}(2)|=O_X(\dim V_{\blambda(n)})$.
\end{proof}
The following proposition allows us to estimate the integral on $\Re(s)=2$ with the right vertical side of the chosen contour.
\begin{proposition}\label{Psiestimate}
\[\Psi_1^{\chi,\blambda(n)}(x)=\frac{1}{2\pi i}\int_{2-iT(n)}^{2+iT(n)}\frac{x^{s+1}}{s(s+1)}D_{\chi,\blambda(n)}(s)ds+O_X\left(\frac{\dim V_{\blambda(n)}x^3}{T(n)}\right).\]
\end{proposition}
\begin{proof}From Proposition~\ref{contourshift}, we have $\Psi_1^{\chi,\blambda(n)}(x)=\frac{1}{2\pi i}\int_{2-i\infty}^{2+i\infty}\frac{x^{s+1}}{s(s+1)}D_{\chi,\blambda(n)}(s)ds$. Consequently,
\begin{eqnarray*}
\left|\Psi_1^{\chi,\blambda(n)}(x)-\frac{1}{2\pi i}\int_{2-iT(n)}^{2+iT(n)}\frac{x^{s+1}}{s(s+1)}D_{\chi,\blambda(n)}(s)ds\right|&\leq&\frac{1}{\pi}\int_{2+iT(n)}^{2+i\infty}\left|\frac{x^{s+1}}{s(s+1)}D_{\chi,\blambda(n)}(s)\right||ds|\\&\leq&
\frac{|D_{\chi,\blambda(n)}(2)|x^{3}}{\pi}\int_{T(n)}^{\infty}\frac{1}{t^2}dt\leq O_X\left(\frac{\dim V_{\blambda(n)}x^3}{T(n)}\right).
\end{eqnarray*}
The last inequality follows from $|D_{\chi,\blambda(n)}(2)|\leq\dim V_{\blambda(n)}|D_{1,0}(2)|=O_X(\dim V_{\blambda(n)})$
\end{proof}
\begin{proposition}For $s=\sigma+iT(n)$, $-1\leq\sigma\leq 2$, asymptotically almost surely as $n\rightarrow\infty$, $\chi:\Gamma\rightarrow U(n)$ satisfy
\begin{equation}
|D_{\chi,\blambda(n)}(s)|=O_X(n^2(\dim V_{\blambda(n)})^2T(n)^2).
\end{equation}
\end{proposition}
\begin{proof}
Choose $\beta\geq T(n)$. Then we have
\begin{eqnarray*}
\frac{D_{\chi,\blambda(n)}(s)}{2s-1}&=&\sum_k\left(\frac{1}{r_k^2+(s-\frac{1}{2})^2}-\frac{1}{r_k^2+\beta^2}\right)-\frac{\dim V_{\blambda(n)}\area(X)}{4\pi}\int_{-\infty}^{\infty}r\left(\frac{1}{r^2+(s-\frac{1}{2})^2}-\frac{1}{r^2+\beta^2}\right)\tanh(\pi r)dr\\&-&\frac{D_{\chi,\blambda(n)}(\frac{1}{2}+\beta)}{2\beta}+\left(\frac{1}{2s-1}-\frac{1}{2\beta}\right)\sum_{\alpha_{hj}}\log|1-e^{2\pi i\alpha_{hj}}|.
\end{eqnarray*}
From Proposition 4.9 of~\cite{Hejhal1},
\begin{equation}\label{intprop49}
\frac{\dim V_{\blambda(n)}\area(X)}{4\pi}\int_{-\infty}^{\infty}r\left(\frac{1}{r^2+(s-\frac{1}{2})^2}-\frac{1}{r^2+\beta^2}\right)\tanh(\pi r)dr=\frac{\dim V_{\blambda(n)}\area(X)}{4\pi}\sum_{k=0}^{\infty}\left(\frac{1}{s+k}-\frac{1}{\beta+\frac{1}{2}+k}\right).
\end{equation}
To bound this quantity, we split the sum into two parts:
\begin{eqnarray*}
\sum_{k=0}^{\infty}\left(\frac{1}{s+k}-\frac{1}{\beta+\frac{1}{2}+k}\right)=\sum_{k=0}^{T(n)}\left(\frac{1}{s+k}-\frac{1}{\beta+\frac{1}{2}+k}\right)+\sum_{k\geq T(n)}\left(\frac{1}{s+k}-\frac{1}{\beta+\frac{1}{2}+k}\right).
\end{eqnarray*}
\[\left|\sum_{k=0}^{T(n)}\left(\frac{1}{s+k}-\frac{1}{\beta+\frac{1}{2}+k}\right)\right|\leq \sum_{k=0}^{T(n)}\frac{1}{|\sigma+k+iT(n)|}+\sum_{k=0}^{T(n)}\frac{1}{T(n)+k}=O(1).\]
On the other hand,
\[\left|\sum_{k\geq T(n)}\left(\frac{1}{s+k}-\frac{1}{\beta+\frac{1}{2}+k}\right)\right|\leq \sum_{k\geq T(n)}\frac{|\beta+\frac{1}{2}-\sigma-iT(n)|}{|\sigma+iT(n)+k||\beta+\frac{1}{2}+k|}=O(T(n))\sum_{k\geq T(n)}\frac{1}{k^2}=O(1).\]
Therefore,
\begin{equation}\label{integralW}
\frac{\dim V_{\blambda(n)}\area(X)}{4\pi}\int_{-\infty}^{\infty}r\left(\frac{1}{r^2+(s-\frac{1}{2})^2}-\frac{1}{r^2+\beta^2}\right)\tanh(\pi r)dr=O_X(\dim V_{\blambda(n)}).
\end{equation}
Letting $r_1,\hdots,r_M\notin\mathbb{R}$ correspond to the eigenvalues less than $\frac{1}{4}$, we also have
\begin{eqnarray}\label{smalleigen}
\sum_{1\leq k\leq M}\left(\frac{1}{r_k^2+(s-\frac{1}{2})^2}-\frac{1}{r_k^2+\beta^2}\right)=O_X(\dim V_{\blambda(n)}).
\end{eqnarray}
Note that $M=O_X(\dim V_{\blambda(n)})$ from~\cite{BallmanPolymerakis}. By Proposition~\ref{logbound}, asymptotically almost surely as $n\rightarrow\infty$, we have
\begin{equation}\label{logs}
\left(\frac{1}{2\beta}-\frac{1}{2s-1}\right)\sum_{\alpha_{hj}}\log|1-e^{2\pi i\alpha_{hj}}|=O_X\left(\frac{n^2\dim V_{\blambda(n)}}{T(n)}\right).
\end{equation}
So far, we have
\begin{equation}
\frac{D_{\chi,\blambda(n)}(s)}{2s-1}=O_X\left(\dim V_{\blambda(n)}\left(1+\frac{n^2}{T(n)}\right)\right)+\sum_{r_k\geq 0}\left(\frac{1}{r_k^2+(s-\frac{1}{2})^2}-\frac{1}{r_k^2+\beta^2}\right).
\end{equation}
We split the latter sum as follows:
\[\sum_{r_k\geq 0}=\sum_{0\leq r_k<T(n)-1}+\sum_{T(n)-1\leq r_k\leq T(n)+1}+\sum_{T(n)+1<r_k\leq 2T(n)}+\sum_{r_k\geq 2T(n)}.\]
Using $N(0\leq r_k\leq T(n))=O_X(n^2\dim V_{\blambda(n)}T(n)^2)$, it is easy to see that
\[\sum_{0\leq r_k<T(n)-1}\left(\frac{1}{r_k^2+(s-\frac{1}{2})^2}-\frac{1}{r_k^2+\beta^2}\right)+\sum_{T(n)+1<r_k\leq 2T(n)}\left(\frac{1}{r_k^2+(s-\frac{1}{2})^2}-\frac{1}{r_k^2+\beta^2}\right)=O_X(n^2\dim V_{\blambda(n)}T(n)).\]
Furthermore, we have
\begin{eqnarray*}
\sum_{r_k\geq 2T(n)}\left|\frac{1}{r_k^2+(s-\frac{1}{2})^2}-\frac{1}{r_k^2+\beta^2}\right|&=& \sum_{r_k\geq 2T(n)}\left|\frac{\beta^2-(\sigma-\frac{1}{2}+iT(n))^2}{(\sigma-\frac{1}{2}+i(T(n)+r_k))(\sigma-\frac{1}{2}+i(T(n)-r_k))(r_k^2+\beta^2)}\right|\\&=&
\sum_{r_k\geq 2T(n)}\frac{O(T(n)^2)}{|\sigma-\frac{1}{2}+i(T(n)+r_k)||\sigma-\frac{1}{2}+i(T(n)-r_k)|r_k^2}\\&=& 
O\left(\sum_{r_k\geq 2T(n)}\frac{T(n)^2}{r_k^4}\right)\\&=&O(T(n)^2)\int_{2T(n)}^{\infty}\frac{dN(t)}{t^4},
\end{eqnarray*}
where $N(t):=N(0\leq r_k\leq t)$. Since $N(t)=O_X(n^2\dim V_{\blambda(n)}t^2)$, we have that the latter is $O_X(n^2\dim V_{\blambda(n)})$.\\
\\
For the final sum, we have 
\begin{eqnarray*}&&\sum_{T(n)-1\leq r_k\leq T(n)+1}\left|\frac{1}{r_k^2+(s-\frac{1}{2})^2}-\frac{1}{r_k^2+\beta^2}\right|\\&=&
\sum_{T(n)-1\leq r_k\leq T(n)+1}\left|\frac{\beta^2-(\sigma-\frac{1}{2}+iT(n))^2}{(\sigma-\frac{1}{2}+i(T(n)+r_k))(\sigma-\frac{1}{2}+i(T(n)-r_k))(r_k^2+\beta^2)}\right|\\&\leq&
O\left(\sum_{T(n)-1\leq r_k\leq T(n)+1}\frac{1}{T(n)|T(n)-r_k|}\right).
\end{eqnarray*}
By~\eqref{admissible}, $|T(n)-r_k|\geq \frac{C_X}{\dim V_{\blambda(n)}T(n)}$. From this and Proposition~\ref{diffeigen}, it follows that
\begin{equation}
O\left(\sum_{T(n)-1\leq r_k\leq T(n)+1}\frac{1}{T(n)|T(n)-r_k|}\right)=O_X\left(\dim V_{\blambda(n)}N(T(n)-1\leq r_k\leq T(n)+1)\right)=O_X(n^2(\dim V_{\blambda(n)})^2T(n)).
\end{equation}
Putting the above together, we obtain the desired bound.
\end{proof}
As a corollary, we bound the integral on the right hand side of equation~\eqref{tophor}.
\begin{corollary}Suppose $0<\varepsilon<\frac{1}{10}$. Then, asymptotically almost surely as $n\rightarrow\infty$, we have
\begin{equation}
\int_{\frac{1}{2}+iT(n)}^{2+iT(n)}\frac{x^{\sigma+1}}{T(n)^2}|D_{\chi,\blambda(n)}(s)||ds|=O_X\left(\varepsilon n^2(\dim V_{\blambda(n)})^2 x^{3/2+2\varepsilon}\right)+O_X\left(C(\varepsilon)\dim V_{\blambda(n)}x^3T(n)^{-4\varepsilon}\right)
\end{equation}
for some constant $C(\varepsilon)>0$.
\end{corollary}
\begin{proof}
We split the integral into two parts:
\[\int_{\frac{1}{2}+iT(n)}^{2+iT(n)}=\int_{\frac{1}{2}+iT(n)}^{\frac{1}{2}+2\varepsilon+iT(n)}+\int_{\frac{1}{2}+2\varepsilon+iT(n)}^{2+iT(n)}.\]
Applying the previous proposition gives us
\[\int_{\frac{1}{2}+iT(n)}^{\frac{1}{2}+2\varepsilon+iT(n)}\frac{x^{\sigma+1}}{T(n)^2}|D_{\chi,\blambda(n)}(s)||ds|=O_X\left(\varepsilon n^2(\dim V_{\blambda(n)})^2 x^{3/2+2\varepsilon}\right).\]
For the second integral, a simple variant of the Phragm\'en--Lindel\"of theorem gives the following variant of Proposition 6.7 of~\cite[p. 103]{Hejhal1}: for $\Re(s)\geq\frac{1}{2}+\varepsilon$ and $\Im(s)$ sufficiently large,
\[|D_{\chi,\blambda(n)}(s)|=O_X\left(\frac{1}{\varepsilon}\dim V_{\blambda(n)}\Im(s)^{2\max\{0,1+\varepsilon-\Re(s)\}}\right).\]
Using this, we obtain as in Proposition 6.13(b) of~\cite[p. 106]{Hejhal1} that
\[\int_{\frac{1}{2}+2\varepsilon+iT(n)}^{2+iT(n)}\frac{x^{\sigma+1}}{T(n)^2}|D_{\chi,\blambda(n)}(s)||ds|=O_X\left(C(\varepsilon)\dim V_{\blambda(n)}x^3T(n)^{-4\varepsilon}\right)\]
for some constant $C(\varepsilon)>0$.
\end{proof}
Combining the previous propositions with equation~\eqref{Cauchyresidue}, we asymptotically almost surely as $n\rightarrow\infty$ have
\begin{align}\label{residueabstract}
\Psi_1^{\chi,\blambda(n)}(x)&=O_X\left(\frac{\dim V_{\blambda(n)}x^3}{T(n)}\right)+O_X\left(\varepsilon n^2(\dim V_{\blambda(n)})^2 x^{3/2+2\varepsilon}\right)+O_X\left(C(\varepsilon)\dim V_{\blambda(n)}x^3T(n)^{-4\varepsilon}\right)\\&+O_X\left(\left(1+\frac{n^2}{T(n)}\right)\frac{x^{3/2}\dim V_{\blambda(n)}}{T(n)}\right)+O_X\left(\left(1+\frac{n^2}{A}\right)\frac{\dim V_{\blambda(n)}x^{1-A}}{A\log x}\right) \nonumber \\&+\sum_{R(A,T(n))}\textup{Res}\left(\frac{x^{s+1}}{s(s+1)}D_{\chi,\blambda(n)}(s)\right)\nonumber
\end{align}
In the following proposition, we estimate the residue part. Recall the notation of Subsection~\ref{spectraltheory}.
\begin{proposition}For any $1/3>\delta>0$, asymptotically almost surely as $n\rightarrow\infty$ we have
\begin{align}\label{residuecomp}
&\sum_{R(\frac{3}{2},T(n))}\textup{Res}\left(\frac{x^{s+1}}{s(s+1)}D_{\chi,\blambda(n)}(s)\right)\\&=\sum_{\frac{1}{4}-\delta<\lambda_k<\frac{1}{4}}\frac{x^{s_k+1}}{s_k(s_k+1)}+\sum_{0\leq r_k\leq T(n)}\frac{x^{s_k+1}}{s_k(s_k+1)}+\sum_{0\leq r_k\leq T(n)}\frac{x^{\tilde{s}_k+1}}{\tilde{s}_k(\tilde{s}_k+1)}+O_X(\dim V_{\blambda(n)}x^{3/2}).\nonumber
\end{align}
\end{proposition}
\begin{proof}
From Theorem 2.16 of~\cite[p.439]{Hejhal2}, $D_{\chi,\blambda(n)}$ has singularities given by the formal expression
\begin{align}
\sum_k\left(\frac{1}{s-s_k}+\frac{1}{s-\tilde{s}_k}\right)+\frac{\dim V_{\blambda(n)}\area(X)}{2\pi}\sum_{k=0}^{\infty}\frac{1-2s}{s+k}.
\end{align}

For each $s_k$, $k\geq 0$, we have
\begin{equation}\label{res1}
\textup{Res}\left(\frac{x^{s+1}}{s(s+1)}D_{\chi,\blambda(n)}(s); s=s_k\right)=\frac{x^{s_k+1}}{s_k(s_k+1)}.
\end{equation}
Futhermore, for each $\tilde{s}_k$, $k\geq 1$, we have
\begin{equation}\label{res2}
\textup{Res}\left(\frac{x^{s+1}}{s(s+1)}D_{\chi,\blambda(n)}(s); s=\tilde{s}_k\right)=\frac{x^{\tilde{s}_k+1}}{\tilde{s}_k(\tilde{s}_k+1)}.
\end{equation}
The residues at $s=-1$ and $s=0$ are calculated as follows. For $s=-1$, we have 
\[D_{\chi,\blambda(n)}(s)=\frac{3\dim V_{\blambda(n)}\area(X)}{2\pi(s+1)}\left(1+e_1(s+1)+e_2(s+1)^2+\dots\right),\]
where the $e_i$ are absolute constants. Furthermore, we have
\[\frac{1}{s}=\frac{-1}{s+1}\left(1+(s+1)+(s+1)^2+\dots\right)\]
and
\[x^{s+1}=1+(s+1)\log x+\frac{((s+1)\log x)^2}{2!}+\dots.\]
From these, it follows that
\[\textup{Res}\left(\frac{x^{s+1}}{s(s+1)}D_{\chi,\blambda(n)}(s); s=-1\right)=\frac{-3\dim V_{\blambda(n)}\area(X)}{2\pi(s+1)}\left(\log x+e_1+1\right).\]
For $s=0$, we have
\[D_{\chi,\blambda(n)}(s)=\frac{1}{s}\left(1+\frac{\dim V_{\blambda(n)}\area(X)}{2\pi}\right)+\frac{\dim V_{\blambda(n)}\area(X)}{2\pi s}\left(f_1s+f_2s^2+\dots\right),\]
where the $f_i$ are absolute constants. We also have
\[\frac{1}{s+1}=1-s+s^2-\dots\]
and
\[x^{s+1}=x\left(1+s\log x+\frac{(s\log x)^2}{2!}+\dots\right).\]
From these, it follows that
\[\textup{Res}\left(\frac{x^{s+1}}{s(s+1)}D_{\chi,\blambda(n)}(s); s=0\right)=\left(\frac{f_1\dim V_{\blambda(n)}\area(X)}{2\pi}-\frac{\dim V_{\blambda(n)}\area(X)}{2\pi}-1\right)x+\left(1+\frac{\dim V_{\blambda(n)}\area(X)}{2\pi}\right)x\log x.\]
Putting the above together, we obtain
\begin{align}\label{residuecomp}
&\sum_{R(\frac{3}{2},T(n))}\textup{Res}\left(\frac{x^{s+1}}{s(s+1)}D_{\chi,\blambda(n)}(s)\right)\\&=\sum_{k=0}^{M}\frac{x^{s_k+1}}{s_k(s_k+1)}+\sum_{k=1}^M\frac{x^{\tilde{s}_k+1}}{\tilde{s}_k(\tilde{s}_k+1)}+\sum_{0\leq r_k\leq T(n)}\frac{x^{s_k+1}}{s_k(s_k+1)}+\sum_{0\leq r_k\leq T(n)}\frac{x^{\tilde{s}_k+1}}{\tilde{s}_k(\tilde{s}_k+1)}+O_X(\dim V_{\blambda(n)}x\log x).\nonumber
\end{align}
By Theorem~\ref{Thm1}, asymptotically almost surely as $n\rightarrow\infty$, the smallest eigenvalue of the flat bundle $E_{\chi,\blambda(n)}$ is at least $\frac{1}{4}-\delta$. Furthermore, $M=O_X(\dim V_{\blambda(n)})$ by~\cite{BallmanPolymerakis}. Consequently,
\[\sum_{k=1}^M\frac{x^{\tilde{s}_k+1}}{\tilde{s}_k(\tilde{s}_k+1)}=O_X(\dim V_{\blambda(n)}x^{3/2}).\]
The conclusion follows.


\end{proof}
From equation~\eqref{residueabstract} applied to $A=\frac{3}{2}$, we therefore have that for any $\varepsilon>0$ and any $\frac{1}{3}>\delta>0$, asymptotically almost surely as $n\rightarrow\infty$,
\begin{align}
\Psi_1^{\chi,\blambda(n)}(x)&=\sum_{\frac{1}{4}-\delta<\lambda_k<\frac{1}{4}}\frac{x^{s_k+1}}{s_k(s_k+1)}+\sum_{0\leq r_k\leq T(n)}\frac{x^{s_k+1}}{s_k(s_k+1)}+\sum_{0\leq r_k\leq T(n)}\frac{x^{\tilde{s}_k+1}}{\tilde{s}_k(\tilde{s}_k+1)}\\&+O_X\left(\frac{\dim V_{\blambda(n)}x^3}{T(n)}\right)+O_X\left(\varepsilon n^2(\dim V_{\blambda(n)})^2 x^{3/2+2\varepsilon}\right)\\&+O_X\left(C(\varepsilon)\dim V_{\blambda(n)}x^3T(n)^{-4\varepsilon}\right)+O_X(\dim V_{\blambda(n)}x^{3/2})\nonumber.
\end{align}
Note that
\begin{align}
&\left|\sum_{0\leq r_k\leq T(n)}\frac{(x+h)^{s_k+1}-x^{s_k+1}}{s_k(s_k+1)}+\sum_{0\leq r_k\leq T(n)}\frac{(x+h)^{\tilde{s}_k+1}-x^{\tilde{s}_k+1}}{\tilde{s}_k(\tilde{s}_k+1)}\right|\nonumber\\&\leq
O(x^{3/2})\sum_{0\leq r_k\leq T(n)}\frac{1}{|s_k(s_k+1)|}+\sum_{0\leq r_k\leq T(n)}\frac{1}{|\tilde{s}_k(\tilde{s}_k+1)|}.\nonumber
\end{align}
Note that
\begin{eqnarray*}
\sum_{0\leq r_k\leq T(n)}\frac{1}{|s_k(s_k+1)|}&=&O_X(n^2\dim V_{\blambda(n)})+\sum_{j=1}^{T(n)-1}\sum_{j\leq r_k<j+1}\frac{1}{|s_k(s_k+1)|}\\&\leq&
O_X(n^2\dim V_{\blambda(n)})+\sum_{j=1}^{T(n)-1}\frac{N(j\leq r_k<j+1)}{j^2}\\&\leq & O_X(n^2\dim V_{\blambda(n)})+\frac{N(0\leq r_k\leq T(n))}{(T(n)-1)^2}-N(0\leq r_k<1)-O(1)\sum_{j=2}^{T(n)-1}\frac{N(0\leq r_k<j)}{j^3}\\&\leq& O_X(n^2\dim V_{\blambda(n)}\log T(n)),
\end{eqnarray*}
where in the last inequality, we used Proposition~\ref{Weylbound} stating that asymptotically almost surely as $n\rightarrow\infty$ we have $N(0\leq r_k\leq T)=O_X(n^2\dim V_{\blambda(n)}T^2)$. We similarly obtain
\[\sum_{0\leq r_k\leq T(n)}\frac{1}{|\tilde{s}_k(\tilde{s}_k+1)|}=O_X(n^2\dim V_{\blambda(n)}\log T(n)).\]
Consequently,
\[\left|\sum_{0\leq r_k\leq T(n)}\frac{(x+h)^{s_k+1}-x^{s_k+1}}{s_k(s_k+1)}+\sum_{0\leq r_k\leq T(n)}\frac{(x+h)^{\tilde{s}_k+1}-x^{\tilde{s}_k+1}}{\tilde{s}_k(\tilde{s}_k+1)}\right|=O_X(n^2\dim V_{\blambda(n)}x^{3/2}\log T(n)).\]
A standard argument shows that if $1\leq h\leq \frac{x}{2}$ and we have the above conditions on $\varepsilon,\delta,$ then asymptotically almost surely as $n\rightarrow\infty$,
\begin{align}
\Psi^{\chi,\blambda(n)}(x)&=\sum_{\frac{1}{4}-\delta<\lambda_k<\frac{1}{4}}\frac{x^{s_k}}{s_k}+O_X\left(\frac{\dim V_{\blambda(n)}x^3}{hT(n)}\right)+O_X\left(\frac{\varepsilon n^2(\dim V_{\blambda(n)})^2 x^{3/2+2\varepsilon}}{h}\right)\\&+O_X\left(\frac{C(\varepsilon)\dim V_{\blambda(n)}x^3T(n)^{-4\varepsilon}}{h}\right)+O_X\left(\frac{n^2\dim V_{\blambda(n)}x^{3/2}\log T(n)}{h}\right)+O_X(\dim V_{\blambda(n)}h)
\end{align}
Let $x=T(n)^{\alpha}$ for some $0<\alpha<1$, and let $h=T(n)^{\frac{3}{4}\alpha+\varepsilon\alpha}$. We obtain
\begin{align}
\Psi^{\chi,\blambda(n)}(T(n)^{\alpha})&=\sum_{\frac{1}{4}-\delta<\lambda_k<\frac{1}{4}}\frac{T(n)^{\alpha s_k}}{s_k}+O_X\left(\dim V_{\blambda(n)}T(n)^{\frac{9}{4}\alpha-1-\varepsilon\alpha}\right)+O_X\left(\varepsilon n^2(\dim V_{\blambda(n)})^2 T(n)^{\frac{3}{4}\alpha+\varepsilon\alpha}\right)\\&+O_X\left(C(\varepsilon)\dim V_{\blambda(n)}T(n)^{\frac{9}{4}\alpha-4\varepsilon-\alpha\varepsilon}\right)+O_X\left(n^2\dim V_{\blambda(n)}T(n)^{\frac{3}{4}\alpha-\varepsilon\alpha}\log T(n)\right)+O_X(\dim V_{\blambda(n)}T(n)^{\frac{3}{4}\alpha+\varepsilon\alpha}).\nonumber
\end{align}
Letting $\alpha=\varepsilon<\frac{1}{10}$, we obtain that asymptotically almost surely as $n\rightarrow\infty$,
\begin{equation}\label{psifin}
\Psi^{\chi,\blambda(n)}(T(n)^{\varepsilon})=O_{X,\varepsilon}\left(n^2(\dim V_{\blambda(n)})^2T(n)^{\left(\frac{3}{4}+\varepsilon\right)\varepsilon}\right).
\end{equation}
An integration by parts argument gives that for every $0<\varepsilon<\frac{1}{10}$, asymptotically almost surely as $n\rightarrow\infty$,
\[\sum_{\ell(\gamma)\leq \varepsilon\log T(n)}\chi_{\blambda(n)}(\chi(\gamma))=O_X\left(n^2\dim V_{\blambda(n)}T(n)^{\left(\frac{3}{4}+\varepsilon\right)\varepsilon}\right).\]
Since by the usual prime geodesic theorem,
\[\sum_{\ell(\gamma)\leq \varepsilon\log T(n)}1\sim \frac{T(n)^{\varepsilon}}{\varepsilon\log T(n)},\]
we obtain, upon replacing $T(n)$ with $T(n)^{1/\varepsilon}$, the following theorem.
\begin{theorem}Suppose $X$ is a cusped hyperbolic surface, and suppose $\blambda(n)$ is sequence of unbalanced signatures with $|\blambda(n)|\leq c\frac{\log n}{\log\log n}$ for large enough $n$, where $c>0$ is the universal constant in Theorem~\ref{ThmBC}. Suppose $T(n)$ is a function of $n$ such that $T(n)\rightarrow\infty$ as $n\rightarrow\infty$. For every $\varepsilon>0$,
\begin{equation}
\mathbb{P}\left[\chi:\left|\frac{1}{\pi_X(\log T(n))}\sum_{\ell(\gamma)\leq \log T(n)}\chi_{\blambda(n)}(\chi(\gamma))\right|\leq n^2\dim V_{\blambda(n)}T(n)^{-\frac{1}{4}+\varepsilon}\right]\xrightarrow{n\rightarrow\infty} 1
\end{equation}
\end{theorem}
It would be interesting to refine this theorem further.
\section{Applications of the spectral theorem}\label{applications}
In this section, we note some applications of Theorem~\ref{Thm1}. Given a signature $\blambda$ associated to an irreducible representations of $U(n)$, we can extend it by zeros to obtain irreducible representations for each $U(m)$, $m\geq n$. Therefore, each signature $\blambda$ gives an infinite family of irreducible representations of $U(n)$ for $n$ large enough.\\
\\
First, by the prime geodesic theorem for flat unitary bundles, the following is an immediate consequence of Theorem~\ref{Thm1}.
\begin{corollary}\label{cor:1}Suppose $X$ is a cusped hyperbolic surface, and let $\blambda_1(n),\hdots,\blambda_N(n)$ be signatures such that for each $i$, $|\blambda_i(n)|\leq c\frac{\log n}{\log \log n}$ for large enough $n$, where $c>0$ is the universal constant of Theorem~\ref{ThmBC}. Let $f_n:=\sum_{i=1}^Na_i\chi_{\blambda_i(n)}$, $a_i\in\mathbb{C}$ fixed, be class functions on $U(n)$ for large enough $n$. Then asymptotically almost surely as $n\rightarrow \infty$, $\chi:\pi_1(X)\rightarrow U(n)$ satisfies
\[\frac{1}{\pi_X(T)}\sum_{\ell(\gamma)\leq T}f_n(\chi(\gamma))=\int_{U(n)}f_n(g)d\mu(g)+O(e^{-T/4})\]
as $T\rightarrow\infty$. Here, $\pi_X(T)$ is the number of closed oriented hyperbolic geodesics $\gamma$ on $X$ of length $\ell(\gamma)\leq T$, and $d\mu$ is the probability Haar measure on $U(n)$.
\end{corollary}
Suppose $r\geq 1$ and $\boldsymbol{a}=(a_1,\hdots,a_r)$ and $\boldsymbol{b}=(b_1,\hdots,b_r)$ are $r$-tuples of non-negative integers. Consider the function
\[\prod_{j=1}^r\tr(g^j)^{a_j}\overline{\tr(g^j)}^{b_j}\]
on $U(n)$. This is a class function, and we claim the following.
\begin{lemma}\label{lem:tracepartition}
There are signatures $\blambda_1,\hdots,\blambda_N$ and constants $a_1,\hdots,a_N\in\mathbb{C}$ depending on the tuples $\boldsymbol{a}$ and $\boldsymbol{b}$ such that
\[\prod_{j=1}^r\tr(g^j)^{a_j}\overline{\tr(g^j)}^{b_j}=\sum_{i=1}^Na_i\dim V_{\blambda_i}\chi_{\blambda_i}\]
as class functions on $U(n)$ for each $n$ large enough.
\end{lemma}
\begin{proof}
It is clear that
\[\tr(g^j)=P_j(\exp(i\theta_1),\hdots,\exp(i\theta_n)),\]
where $\exp(i\theta_j)$ are the eigenvalues of $g\in U(n)$. It follows from this that
\[\prod_{j=1}^r\tr(g^j)^{a_j}\overline{\tr(g^j)}^{b_j}=P_{\lambda}(g)\overline{P_{\mu}(g)},\]
where $\lambda=1^{a_1}2^{a_2}\dots r^{a_r}$ and $\mu=1^{b_1}2^{b_2}\dots r^{b_r}$. However, letting $K=a_1+2a_2+\hdots+ra_r$ and $L=b_1+2b_2+\hdots+rb_r$, we have
\[P_{\lambda}=\sum_{\nu_1\vdash K}\chi_{\lambda}(\nu_1)s_{\nu_1}\]
and
\[P_{\mu}=\sum_{\nu_2\vdash L}\chi_{\mu}(\nu_2)s_{\nu_2},\]
where $\chi_{\lambda}$ and $\chi_{\mu}$ are characters of the symmetric groups $S_K$ and $S_L$, respectively. Therefore,
\[\prod_{j=1}^r\tr(g^j)^{a_j}\overline{\tr(g^j)}^{b_j}=\sum_{\substack{\nu_1\vdash K\\ \nu_2\vdash L}}\chi_{\lambda}(\nu_1)\overline{\chi_{\mu}(\nu_1)}s_{\nu_1}(g)s_{\nu_2}(g^{-1}).\]
We may then write the class functions $s_{\nu_1}(g)s_{\nu_2}(g^{-1})$ as linear combinations of irreducible characters of unitary groups $U(n)$, for large enough $n$, depending only of partitions independent of $n$. The conclusion follows.
\end{proof}
Lemma~\ref{lem:tracepartition} combined with Corollary~\ref{cor:1} and the main result of Diaconis--Shahshahani~\cite{DS} give the following.
\begin{corollary}\label{cor:2}Suppose $r\geq 1$ and $\boldsymbol{a}=(a_1,\hdots,a_r)$ and $\boldsymbol{b}=(b_1,\hdots,b_r)$ are $r$-tuples of non-negative integers. Then for $X$ any cusped hyperbolic surface, asymptotically almost surely as $n\rightarrow\infty$, $\chi:\pi_1(X)\rightarrow U(n)$ satisfies
\[\frac{1}{\pi_X(T)}\sum_{\ell(\gamma)\leq T}\prod_{j=1}^r\tr(\chi(\gamma^j))^{a_j}\overline{\tr(\chi(\gamma^j))}^{b_j}=\delta_{\boldsymbol{a},\boldsymbol{b}}\int_{\mathbb{C}^r}|z_1|^{2a_1}\hdots|z_r|^{2a_r}\prod_{j=1}^r\frac{1}{j}e^{-\pi|z_j|^2/j}dx_jdy_j+O(e^{-T/4})\]
as $T\rightarrow\infty$, where $\delta_{\boldsymbol{a},\boldsymbol{b}}$ is the Kronecker delta.
\end{corollary}
The main term on the right hand side is the expectation of $\prod_{j=1}^r\tr(g^j)^{a_j}\overline{\tr(g^j)}^{b_j}$ over $U(n)$. Explicitly, we have
\[\int_{\mathbb{C}^r}|z_1|^{2a_1}\hdots|z_r|^{2a_r}\prod_{j=1}^r\frac{1}{j}e^{-\pi|z_j|^2/j}dx_jdy_j=\prod_{j=1}^rj^{a_j}a_j!.\]

\section{Comments on the balanced case}\label{comments}
In the situation where the signatures are balanced, the hyperbolic scattering determinant appears in a serious way in the computations of Sections~\ref{counting} and~\ref{pgt} proving Theorem~\ref{geodesicdep}. Therefore, we need to be able to understand it, at least from a probabilistic perspective. We recall the hyperbolic scattering determinant and discuss two problems that arise.\\
\\
Aside from the discrete spectrum of the Laplacian $\Delta_{\chi}$, there is often a continuous spectrum corresponding to the cusps of $X$. Let $A_{\chi}=\widehat{\bigoplus}_n\mathbb{C}\varphi_n$ be the subspace of $L^2(X;E_{\chi})$ generated by the eigensections $\varphi_n$ of $\Delta_{\chi}$, and let $\cal{E}_{\chi}$ be the orthogonal complement of $A_{\chi}$ in $L^2(X;E_{\chi})$. It may be described in terms of what are called Eisenstein series. All in all, we have a spectral decomposition
\[L^2(X;E_{\chi})=A_{\chi}\oplus \cal{E}_{\chi}.\]
In the spectral theory of hyperbolic surfaces, when $K_0\geq 1$, that is, when there is a continuous part, the hyperbolic scattering determinant plays an important role. This is defined as follows. For each of the parabolic $T_j$ corresponding to the $j^{\textup{th}}$ cusp, choose an orthonormal basis $\boldsymbol{f}_{1j},\hdots,\boldsymbol{f}_{nj}$ of $\mathbb{C}^n$ such that
\begin{equation}\label{basis}\chi(T_j^{-1})\boldsymbol{f}_{hj}=e^{2\pi i\alpha_{hj}}\boldsymbol{f}_{hj},\ h=1,\hdots,n,\end{equation}
where $\alpha_{hj}$ are real numbers. For each $j$, we may assume without loss of generality that $\alpha_{hj}=0$ for $h=1,\hdots,m_j$. Choose $N_j\in\PSL_2(\mathbb{R})$ such that 
\begin{equation}N_j\eta_j=\infty,\ S^{-1}=N_jT_jN_j^{-1},\end{equation}
where $S:=\begin{bmatrix}1 & 1\\ 0& 1\end{bmatrix}$. Then for each $1\leq h\leq m_j$, $1\leq p\leq m_k$, we let
\[\varphi_{hj,pk}(s):=\sqrt{\pi}\frac{\Gamma(s-\frac{1}{2})}{\Gamma(s)}\sum_{T\in \Gamma_j\setminus\Gamma/\Gamma_k}\frac{1}{|c|^{2s}}\boldsymbol{f}^t_{hj}\chi(T^{-1})\overline{\boldsymbol{f}}_{pk},\]
where we omit $T\in \Gamma_j$ if $j=k$, and $c$ comes from writing $N_jTN_k^{-1}$ as $\pm\begin{bmatrix}a& b\\c& d\end{bmatrix}\in\PSL_2(\mathbb{R})$. The hyperbolic scaterring matrix is the $K_0\times K_0$ matrix given by 
\[\Phi_{\chi}(s):=\left[\varphi_{hj,pk}(s)\right]\]
for $\Re(s)>1$, and its determinant $\varphi_{\chi}(s):=\det\Phi_{\chi}(s)$ is called the \textit{hyperbolic scattering determinant}. $\varphi_{\chi}$ has a meromorphic continuation to $\mathbb{C}$. Its poles in $\Re(s)\geq\frac{1}{2}$ have residues that are eigensections of $E_{\chi}$ with eigenvalues in $(0,\frac{1}{4})$ corresponding to $s_1,\hdots,s_{M_e}\in(\frac{1}{2},1)$. We also let $s_0=1$ if $0$ is in the spectrum of $E_{\chi}$; almost surely, this is not the case. This corresponds to the \textit{residual spectrum} of $E_{\chi}$. $\varphi_{\chi}$ also has a distinguished finite set of zeros $\rho_1,\hdots,\rho_N\in [\frac{1}{2},\infty)$ as well as complex zeros $\rho=\frac{1}{2}+\eta+i\gamma$, $\eta\geq 0$, $\gamma>0$, as well as their conjugates $\overline{\rho}$. It is clear that we may write $\varphi_{\chi}$ as a Dirichlet series in the form
\begin{equation}\label{hsdd}
\varphi_{\chi}(s)=\left(\sqrt{\pi}\frac{\Gamma(s-\frac{1}{2})}{\Gamma(s)}\right)^{K_0}\sum_{n=1}^{\infty}\frac{a_n(\chi)}{q_n(\chi)^{2s}}.
\end{equation}
When $X=\SL_2(\mathbb{Z})\setminus\mathbb{H}$ is the modular curve and the bundles are trivial rank $1$ bundles, then
\[\varphi_1(s)=\sqrt{\pi}\frac{\Gamma(s-\frac{1}{2})\zeta(2s-1)}{\Gamma(s)\zeta(2s)}.\]
Using the notation introduced so far, there is another way of writing the hyperbolic scattering determinant $\varphi_{\chi}$~\cite[eq. 2.10, p.437]{Hejhal2}:
\begin{equation}\label{phiprod}
\varphi_{\chi}(s)=(q_1(\chi))^{1-2s}\varphi(1/2)\frac{(s-\rho_1)\dots(s-\rho_N)}{(1-s-\rho_1)\dots(1-s-\rho_N)}\prod_{k=0}^{M_e}\frac{1-s-s_k}{s-s_k}\prod_{\gamma>0}\frac{(s-\rho)(s-\overline{\rho})}{(s+\rho-1)(s+\overline{\rho}-1)}.
\end{equation}
The hyperbolic scattering determinant is closely related to Eisenstein series. For details, the reader may consult Chapter 8 of Hejhal's~\cite{Hejhal2}.\\
\\
We would like to show that generically, some analytic properties of the hyperbolic scattering determinant may be controlled by the surface $X$ and rank of the bundle only. We prove the following proposition having to do with the main term of the hyperbolic scattering determinant for large positive real $s$. In the following, we use the notation above for the surface representation $\pi_{\blambda(n)}\circ\chi$, not $\chi$ itself.
\begin{proposition}\label{genericq1}With probability $1$ over the space of surface representations $\chi:\pi_1(X)\rightarrow U(n)$, $\pi_{\blambda(n)}\circ\chi:\pi_1(X)\rightarrow\GL(V_{\blambda(n)})$ leads to a flat unitary bundle $E_{\chi,\blambda(n)}$ whose hyperbolic scattering determinant has Dirichlet expansion~\eqref{hsdd} satisfying $q_1(\pi_{\blambda(n)}\circ\chi)=q_1(1)^{\dim V_{\blambda(n)}}$, where $q_1(1)$ comes from the hyperbolic scattering determinant of the trivial rank $1$ bundle over $X$.
\end{proposition}
\begin{proof}
Using the notation of~\eqref{basis} for the surface representation $\pi_{\blambda(n)}\circ\chi$, consider for each $j=1,\hdots,k$, $k$ being the number of cusps of $X$, the $\dim V_{\blambda(n)}\times m_j$ matrices
\begin{equation}\label{blockmatrix}
F_j:=\begin{bmatrix}|& &|\\ \boldsymbol{f}_{1j}&\dots&\boldsymbol{f}_{m_jj}\\|& &|\end{bmatrix}
\end{equation}
with orthonormal column vectors the $\boldsymbol{f}_{ij}\in\mathbb{C}^{\dim V_{\blambda(n)}}$, $i=1,\dots,m_j$. Putting these into one matrix, we have the block diagonal $\dim V_{\blambda(n)}\times K_0$ matrix
\begin{equation}
\boldsymbol{F}:=\diag(F_1,\hdots,F_k).
\end{equation}
Consider also the block matrix 
\begin{equation}
\Theta_{\chi}(s):=\left[\sum_{W\in [T_i]\setminus\Gamma/[T_j]}\frac{\pi_{\blambda(n)}(\chi(W^{-1}))}{|c|^{2s}}\right]_{i,j=1,\hdots,k}.
\end{equation}
By definition, the hyperbolic scattering matrix is
\begin{equation}
\Phi_{\pi_{\blambda(n)}\circ\chi}(s)=\left(\sqrt{\pi}\frac{\Gamma(s-\frac{1}{2})}{\Gamma(s)}\right)\boldsymbol{F}^t\Theta_{\chi}(s)\overline{\boldsymbol{F}}.
\end{equation}
It follows from the orthonormality of $\boldsymbol{f}_{1j},\hdots,\boldsymbol{f}_{m_jj}$ that the hyperbolic scattering determinant may be written as
\begin{equation}
\varphi_{\pi_{\blambda(n)}\circ\chi}(s)=\det(\Phi_{\pi_{\blambda(n)}\circ\chi}(s))=\left(\sqrt{\pi}\frac{\Gamma(s-\frac{1}{2})}{\Gamma(s)}\right)^{K_0}\det\Theta_{\chi}(s).
\end{equation} 

As noted in equation~\eqref{hsdd}, we have an expansion of the form
\begin{equation}\label{expansionbundle}\varphi_{\pi_{\blambda(n)}\circ\chi}(s)=\left(\sqrt{\pi}\frac{\Gamma\left(s-\frac{1}{2}\right)}{\Gamma(s)}\right)^{K_0}\sum_{j=1}^{\infty}\frac{a_j(\pi_{\blambda(n)}\circ\chi)}{q_j(\pi_{\blambda(n)}\circ\chi)^{2s}},
\end{equation}
where $q_1(\pi_{\blambda(n)}\circ\chi)<q_2(\pi_{\blambda(n)}\circ\chi)<\hdots$. By definition, $q_1(1)$ is the smallest possible value of $|c|$ that could appear when we consider the trivial rank $1$ bundle over $X$. We want to show that with probability $1$ in the space of surface representations $\chi:\pi_1(X)\rightarrow U(n)$, the $q_1(\pi_{\blambda(n)}\circ\chi)$ satisfy $q_1(\pi_{\blambda(n)}\circ\chi)=q_1(1)^{\dim V_{\blambda(n)}}$. Note that when $\chi=1$ is the trivial representation, we obtain
\begin{eqnarray*}\det\Theta_1(s)&=&\det\left[\sum_{\substack{W\in [T_i]\setminus\Gamma/[T_j]}}\frac{\pi_{\blambda(n)}(I_n)}{|c|^{2s}}\right]_{i,j=1,\hdots,k}\\&=&\det \left(I_{\dim V_{\blambda(n)}}\otimes \left[\sum_{\substack{W\in [T_i]\setminus\Gamma/[T_j]}}\frac{1}{|c|^{2s}}\right]_{i,j=1,\hdots,k}\right)\\&=& \left(\det \left[\sum_{\substack{W\in [T_i]\setminus\Gamma/[T_j]}}\frac{1}{|c|^{2s}}\right]_{i,j=1,\hdots,k}\right)^{\dim V_{\blambda(n)}}\\&=&
\left(\sum_{j=1}^{\infty}\frac{a_j(1)}{q_j(1)^{2s}}\right)^{\dim V_{\blambda(n)}},
\end{eqnarray*}
where the latter Dirichlet series corresponds to the trivial representation. Its main term is $(a_1(1)/q_1(1)^{2s})^{\dim V_{\blambda(n)}}$ as $s\rightarrow+\infty$. Denote by $G(\chi)$ the coefficient of $q_1(1)^{-2s\dim V_{\blambda(n)}}$ in $\det\Theta_{\chi}(s)$. The computation above shows that $G(1)\neq 0$, and so $G\neq 0$. Composing with the exponential map
\[\exp:\mathfrak{u}(n)^k\rightarrow U(n)^k\] 
gives us a non-zero real analytic function
\[\mathbb{R}^{kn^2}\simeq\mathfrak{u}(n)^k\rightarrow\mathbb{C}.\]
Since this is a non-zero real-analytic function, and real-analytic functions vanishing on a set of positive measure are identically zero, $G$ must be non-zero almost everywhere. Consequently, $q_1(\pi_{\blambda(n)}\circ\chi)=q_1(1)^{\dim V_{\blambda(n)}}$ with probability $1$ on the space of surface representations $\chi$.
\end{proof}
Though we do have a good understanding of $q_1(\pi_{\blambda(n)}\circ\chi)$, controlling $|a_1(\pi_{\blambda(n)}\circ\chi)|$ in a uniform manner is much more complicated and is connected to very deep questions in random matrix theory. The control on this term is necessary when one wants to bound sums of the form
\[\sum_{0<\gamma\leq T(n)}\eta.\]
In this case, we want $|a_1(\pi_{\blambda(n)}\circ\chi)|$ asymptotically almost surely as $n\rightarrow\infty$ to not be too small. Such sums would appear in the calculations of Section~\ref{pgt} when $\blambda(n)$ are balanced and so our bundles are singular. For example, asymptotically almost surely having $|a_1(\pi_{\blambda(n)}\circ\chi)|\geq 2^{-(\dim V_{\blambda(n)})^{1+\delta}}$, $\delta>0$, would be more than sufficient. Even simpler questions in random matrix theory are notoriously difficult and very little is known. We discuss this below.\\
\\
Suppose $P(U_1,\hdots,U_k,U_1^*,\hdots,U_k^*)$ is a non-commutative non-zero $*$-polynomial with complex coefficients on unitary matrices of the same size. As usual, we endow $U(n)^k$ with the probability Haar measure $d\mu$. We study the following question.
\begin{question}For each $P$ and each non-trivial signature $\blambda$, is there a function $f_{P,\blambda}(n)$ depending on $P$ and $\blambda$ such that
\[\mathbb{P}\left[(U_1,\hdots,U_k)\in U(n)^k:\ |\det(P(\pi_{\blambda}(U_1),\hdots,\pi_{\blambda}(U_k)))|>\frac{1}{n^{f_{P,\blambda}(n)}}\right]\xrightarrow{n\rightarrow\infty}1?\]
\end{question}
We conjecture the following.
\begin{conjecture}\label{Conj}There is a constant $c=c(P,\blambda)>0$ depending only on $P$ and $\blambda\neq 0$ such that
\begin{equation}\label{(*)}\mathbb{P}\left[(U_1,\hdots,U_k)\in U(n)^k:\ \sigma_{\textup{min}}(P(\pi_{\blambda}(U_1),\hdots,\pi_{\blambda}(U_k)))>\frac{1}{n^{c}}\right]\xrightarrow{n\rightarrow\infty}1\tag{$\boldsymbol{*}$}
\end{equation}
where $\sigma_{\textup{min}}$ denotes the least singular value.
\end{conjecture}
From this conjecture, it would follow that $f_{P,\blambda}(n)=c(\blambda,P)\dim V_{\blambda}$ suffices. In the following proposition, we prove a very strong conclusion for the determinant of the sum of independent Haar unitaries.
\begin{proposition}
Let $P(U_1,\hdots,U_k):=U_1+\hdots+U_k$, $k\geq 1$. For any $0<\delta<1$, we can take $f_{P,(1,0,\hdots)}(n)=-\frac{nc_{\delta}(k)}{\log n}$, where $c_{\delta}(k)=\frac{1}{2}\log\left((1-\delta)k\left(\frac{k-1}{k}\right)^{k-1}\right)$. For sufficiently large $k$, depending on $\delta$, $c_{\delta}(k)>0$.
\end{proposition}
\begin{proof}
By the main result of Basak--Dembo~\cite{BD}, the limiting (weakly in probability) spectral distribution of $P$ is the Brown measure of the free sum of $k$ Haar unitaries. This was shown by Haagerup~\cite[Ex. 5.5]{Haagerup} to be the measure in the complex plane having density
\[\rho(z):=\frac{1}{\pi}\frac{k^2(k-1)}{(k^2-|z|^2)^2}\chi_{|z|\leq\sqrt{k}},\]
where $\chi_{|z|\leq\sqrt{k}}$ is the characteristic function with support the disk centered at $0$ of radius $\sqrt{k}$. Letting $\lambda_i(A)$ denote the eigenvalues of a matrix $A$, we obtain for every $\varepsilon>0$,
\[\mathbb{P}_{U(n)^k}\left[(U_1,\hdots,U_k)\in U(n)^k:\left|\frac{1}{n}\sum_{i=1}^n\log|\lambda_i(P(U_1,\hdots,U_k)|-\frac{k^2(k-1)}{\pi}\int_{|z|\leq\sqrt{k}}\frac{\log|z|}{(k^2-|z|^2)^2}dxdy\right|>\varepsilon\right]\xrightarrow{n\rightarrow\infty}0.\]
However,
\[\frac{k^2(k-1)}{\pi}\int_{|z|\leq\sqrt{k}}\frac{\log|z|}{(k^2-|z|^2)^2}dxdy=\frac{1}{2}\log\left(k\left(\frac{k-1}{k}\right)^{k-1}\right).\]
This implies that for any $0<\delta<1$, we have
\[\mathbb{P}_{U(n)^k}\left[|\det(P)|\geq \left((1-\delta)k\left(\frac{k-1}{k}\right)^{k-1}\right)^{n/2}\right]\xrightarrow{n\rightarrow\infty}0.\]
The conclusion follows.
\end{proof}
Another question that arises is the following.
\begin{question}Is there an upper bound in terms of $X$ and $\dim V_{\blambda(n)}$ on the real parts of the zeros of the hyperbolic scattering determinants $\varphi_{\pi_{\blambda(n)}\circ\chi}$ that is true asymptotically almost surely over $\chi$ as $n\rightarrow\infty$?
\end{question}
This is relevant to the generalization of Proposition~\ref{Weylbound} to the balanced case.

\end{document}